\documentclass[11pt]{amsart}
\usepackage{amssymb, amsmath, amsthm}
\usepackage{color}
\usepackage{graphicx}
\usepackage[final]{hyperref}
\usepackage[a4paper, centering]{geometry}
\usepackage{pdfsync}

\usepackage[mathscr]{eucal}
\usepackage{lmodern}

\newif\ifextended
\extendedfalse

\newtheorem{teor}{Theorem}[section]
\newtheorem{prop}[teor]{Proposition}

\theoremstyle{definition}
\newtheorem{Definition}[teor]{Definition}
\theoremstyle{remark}
\newtheorem{rk}[teor]{Remark}

\long\def\elimina#1{}

\def\R{\mathbb{R}}

\def\N{\mathbb{N}}
\def\Z{\mathbb{Z}}

\def\haus{\mathcal{H}^{1}}
\def\leb{\mathcal{L}^{2}}

\def\Lip{\textrm{Lip}}

\def\ne{\nu^E}
\def\curf{\kappa_\varphi}

\def\lum{\dfrac{\ell_1}{2}}
\def\ldm{\dfrac{\ell_2}{2}}

\def\dist{d_\varphi}

\def\geps{g_\varepsilon}

\def\oray#1{{]\!]#1[\![}}
\def\cray#1{{[\![#1]\!]}}

\DeclareMathOperator{\inte}{int} 
 \DeclareMathOperator{\dive}{div}

\title[Crystalline evolutions with oscillating forcing terms]%
{Crystalline evolutions with rapidly oscillating forcing terms}%
\author[A.~Braides]{Andrea Braides$^\dagger$}
\thanks{$^\dagger$ Dipartimento di Matematica, Università di Roma ``Tor 
Vergata'', via della Ricerca scientifica, 00133 Roma, Italy, email: 
braides@mat.uniroma2.it }
 
\author[A.~Malusa]{Annalisa Malusa$^\ddagger$}
\thanks{$^\ddagger$ Dipartimento di Matematica ``G.\ Castelnuovo'', Sapienza 
Università di Roma, Piazzale Aldo  Moro 2, 00185 Roma, Italy, email: 
malusa@mat.uniroma1.it}

\author[M.~Novaga]{Matteo Novaga$^\S$}
\thanks{$^\S$ Dipartimento di Matematica, Università di Pisa, Largo Pontecorvo 
5, 56216 Pisa, Italy, email: matteo.novaga@unipi.it}

\date{\today}
\begin{document}

\begin{abstract}
We consider the evolution by crystalline curvature of a planar set in a stratified medium,
modeled by  a periodic forcing term. 
We characterize the limit evolution law as the period of the oscillations tends to zero. 
Even if the model is very simple, the limit evolution problem is quite rich, and 
we discuss some properties 
such as uniqueness, comparison principle and pinning/depinning phenomena. 
\end{abstract}

\maketitle

\tableofcontents


\section{Introduction}

In this paper we are interested in the asymptotic behaviour of motions of planar curves according to the law
\begin{equation}\label{eqevol}
v = \curf + g\left(\frac{x}{\varepsilon}\right),
\end{equation}
where $v$ is the normal velocity, $\curf$ is the crystalline square curvature 
(see Definition \ref{d:ch} for precise definitions), $g:\R\to \R$ is a forcing term depending only on the horizontal variable $x$, and $\varepsilon>0$ is a small parameter modeling the rapidly oscillating medium where the curve evolves. For simplicity, we shall assume that $g$ is 1-periodic and takes only two values $\alpha<0<\beta$, with average $\frac{\alpha +\beta}{2}$.

Crystalline evolutions play an important role in many models of phase 
transitions and Materials Science
(see \cite{Gu,Ta} and references therein) and have been significantly studied in recent years
(see for instance \cite{AT,GG,BNP1,BNP2,Gi}). The term $g\left(\frac{x}{\varepsilon}\right)$
models a heterogeneous layered medium, which we assume periodic exactly in one of 
the direction orthogonal to the Wulff shape of the crystalline perimeter (in
our case, for simplicity, the $x$-direction). 
Our aim is understanding the effect of the oscillations in the asymptotic limit $\varepsilon \to 0$,
which is a typical homogenization problem. 

We will show that, at scale $\varepsilon$, the curves may undergo a microscopic 
`facet breaking' phenomenon, with small segments of length proportional 
to $\varepsilon$ being created and reabsorbed.  After dealing with this aspect, 
we show that the motion of a limit of curves satisfying 
\eqref{eqevol} can be characterized by different laws of motion in the $x$ 
and $y$ directions: the portions of the curve moving in the vertical direction 
travel with a velocity equal to $\curf + \frac{\alpha+\beta}{2}$,
whereas the portions moving in the horizontal direction are either
pinned or travel with velocity 
equal to $\langle \curf + g\rangle$, where $\langle \cdot\rangle$
denotes the harmonic mean (see Theorem \ref{t:effpoly}).
Note that such law does not correspond to a forced crystalline evolution.

Our analysis can be set in a large class of variational evolution problems
dealing with limits of motions driven by functionals $F_\varepsilon$ depending
on a small parameter \cite{Br}. In some cases, the 
limit motion can be directly related to the $\Gamma$-limit of $F_\varepsilon$
(see, e.g., \cite{SS,CG,BCGS,BL}), but in general this is not the case.
For oscillating functionals, the energy landscape of the energies $F_\varepsilon$
can be quite different from that of their $\Gamma$-limit and the related
motions can be influenced by the presence of local minima which may give rise
to `pinning' phenomena (motions may be attracted by those local minima) or
to effective homogenized velocities \cite{NV,BGN}. In the case of geometric motions, a 
general understanding of the effects of microstructure is still missing.
Recently, some results have been obtained for (two-dimensional) crystalline
energies, for which a simpler description can be sometimes given in terms of a system
of ODEs. Such results include discrete approximations of crystalline energies,
which can be understood as a simple way to introduce a periodic dependence, corresponding
to that of an underlying square lattice \cite{BGN,BSc,BCY,BSo}. 
Such discrete energies correspond to continuum inhomogeneous perimeter 
energies of the form
\[
F_\varepsilon(E)=\int_{\partial E} a\Bigl(\frac{x}{\varepsilon},\frac{y}{\varepsilon}\Bigr)
\bigl( |\nu_1^E|+|\nu_2^E|\bigr) \, d\haus,\qquad E\subset\R^2,
\]
converging to the square perimeter in the sense of $\Gamma$-convergence \cite{GCB,AB},
and the corresponding geometric motions can be studied using the minimizing-movement 
approach introduced by Almgren, Taylor and Wang \cite{ATW}. The suitably defined limit motions \cite{Br}
correspond to modified crystalline flows. 

In our case, equation \eqref{eqevol} corresponds to the $L^2$-gradient flow for the
energy functional  
\[
F_\varepsilon(E)=\int_{\partial E} 
\bigl( |\nu_1^E|+|\nu_2^E|\bigr) \, d\haus +\int_E 
g\Bigl(\frac{x}{\varepsilon}\Bigr)\ 
d\leb,\qquad E\subset\R^2,
\]
where we identify the evolving curve with the boundary of a set $E$.
Notice that, since the volume term converges to 
$\frac{\alpha +\beta}{2}\leb(E)$, the ($\Gamma$-)limit as $\varepsilon \to 0$ of the 
functionals $F_\varepsilon$ is the functional
\[
\overline{F}(E)=\int_{\partial E} 
\bigl( |\nu_1^E|+|\nu_2^E|\bigr)\, d\haus +\dfrac{\alpha +\beta}{2}\leb(E).
\]
As a consequence of our analysis, it turns out that the 
asymptotic behavior as
$\varepsilon \to 0$ of the evolutions corresponding to \eqref{eqevol}
does not coincide with the gradient flow of 
$\overline{F}$, which is $v=\curf + \frac{\alpha+\beta}{2}$.
Even in the case $\alpha+\beta=0$ the limit motion cannot be derived as a
 gradient flow of a modified (crystalline) perimeter.

Note that, if the crystalline curvature is replaced by the usual curvature and the evolving curve is a unbounded 
graph in the vertical direction, so that it never travels horizontally, the analogous homogenization
problem have been studied in \cite{CNV}, where it is proved that the limit evolution law is indeed 
the curvature flow with a constant forcing term. However, if the curve is a graph in the horizontal direction the problem is more complicated and a general convergence result is still not available.
In this respect, our work is a step in that direction.
We mention, however, that a complete analysis of the asymptotic behavior of the semlinear problem
\[
u_t = u_{xx} +  g\left(\frac{u}{\varepsilon}\right),
\]
can be found in \cite{CDN}. That problem has some features in common with our geometric evolution,
as can be seen as the linearization of the isotropic version of \eqref{eqevol}.

\smallskip

The plan of the paper is the following: in Section \ref{s:sett} we introduce
the notion of crystalline curvature and the evolution problem we want to study. 
In Section \ref{s:calibrability} we introduce
the notion of calibrable edge, that is, an edge of the curve which does not break during the evolution, 
and we characterize the calibrability of an edge in terms of its length and of the position of its endpoints. Finally, in Section \ref{s:eff} we characterize the limit evolution law as 
$\varepsilon\to 0$ first for rectangular sets, then for polyrectangles, and eventually for more general sets, including bounded convex sets.


\section{Setting of the problem}\label{s:sett}

\subsection{Notation}

Given $\xi$ $\eta\in \R^2$, we denote by $\xi\cdot\eta$ the usual scalar 
product between $\xi$ and $\eta$ and by $\oray{\xi,\eta}$ (respectively, $\cray{\xi,\eta}$)
the open (respectively, closed) segment joining $\xi$ with $\eta$. 
The canonical basis of $\R^2$ will be denoted by $e_1=(1,0)$, $e_2=(0,1)$.
 
The 1--dimensional 
Hausdorff measure and the 2--dimensional Lebesgue measure in $\R^2$ will be 
denoted by $\haus$ and $\leb$, respectively.

We say that a set $E\subseteq \R^2$ is a  {\em  Lipschitz set} if $E$ is open and $\partial 
E$  can be written, locally, as the graph of a Lipschitz function 
(with respect to a suitable orthogonal coordinate system). 
The  {\em  outward normal} to $\partial E$ at $\xi$, that exists $\haus$--almost everywhere on 
$\partial E$, will be denoted by  $\ne(\xi)=(\nu_1^E,\nu_2^E)$.

The Hausdorff distance between the two sets
$E,\ F\in \R^2$ will be denoted by $d_H(E,F)$.

\subsection{The crystalline square curvature}

We briefly recall how to give a notion of mean curvature $\kappa^E$ 
which is 
consistent with the requirement that a geometric evolution $E(t)$ that reduces 
as fast as possible the energy functional
\[
P(E)= \int_{\partial E} 
\bigl(|\nu_1^E|+|\nu_2^E|\bigr)\, d\haus
\]
has normal velocity $\kappa^{E(t)}$ 
$\haus$--almost everywhere on $\partial E(t)$. 

The functional $P(E)$ turns out to be the perimeter associated to
the norm 
$\varphi(x,y)=\max\{|x|,|y|\}$, $(x,y)\in\R^2$, that is  the Minkowski content 
derived by considering $(\R^2,\varphi)$ as a normed space. The density
$\varphi^\circ(x,y)=|x|+|y|$ is the polar function of $\varphi$, defined by 
$\varphi^\circ(\xi^\circ):=\sup \{\xi\cdot\xi^\circ,\ \varphi(\xi)\leq 1\}$.

The sublevel sets (or Wulff shapes) $\{\phi(\xi)\leq 1\}$ and 
$\{\phi^\circ(\xi)\leq 1\}$ are the 
square 
$K=[-1,1]^2$, and the square with corners at $(\pm 1,0)$ and $(0,\pm 1)$, 
respectively.

Given a nonempty compact set $E\subseteq \R^2$, we denote by $\dist^E$ the 
{\em oriented  $\varphi$--distance function} to $\partial E$, negative inside 
$E$, that 
is
\[
\dist^E(\xi)=\inf_{\eta\in E}\varphi(\xi-\eta) - \inf_{\eta\not\in  
E}\varphi(\xi-\eta), 
\qquad \xi\in\R^2.
\]
The function $\dist^E$ is Lipschitz, $\varphi^\circ(\nabla 
\dist^E(\xi))=1$ at each its differentiability point $\xi$, and 
\[
\nabla \dist^E(\xi)=\frac{\nu_E(\xi)}{ \varphi^\circ(\nu_E(\xi))}
\]
at every $\xi\in\partial E$ where $\nu_E$ is well defined.

Due to the lack of uniqueness of the projection on $\partial E$ with respect to 
the 
distance $\dist^E$, the intrinsic normal direction is not uniquely determined, 
in general, even if the set $E$ is smooth. It is known that the normal cone at 
$\xi\in\partial E$ is well defined whenever $\xi$ is a differentiability point 
for $\dist^E$ and it is given by $T_{\phi^\circ}\bigl(\nabla\dist^E(\xi)\bigr)$,
where 
\[
T_{\varphi^\circ}(\xi^\circ) :=\{\xi\in \R^2, 
\xi\cdot\xi^\circ=(\varphi^\circ(\xi))^2\}, \quad \xi^\circ\in\R^2\,.
\]
\begin{figure}[h!]
\includegraphics[height=3.5cm]{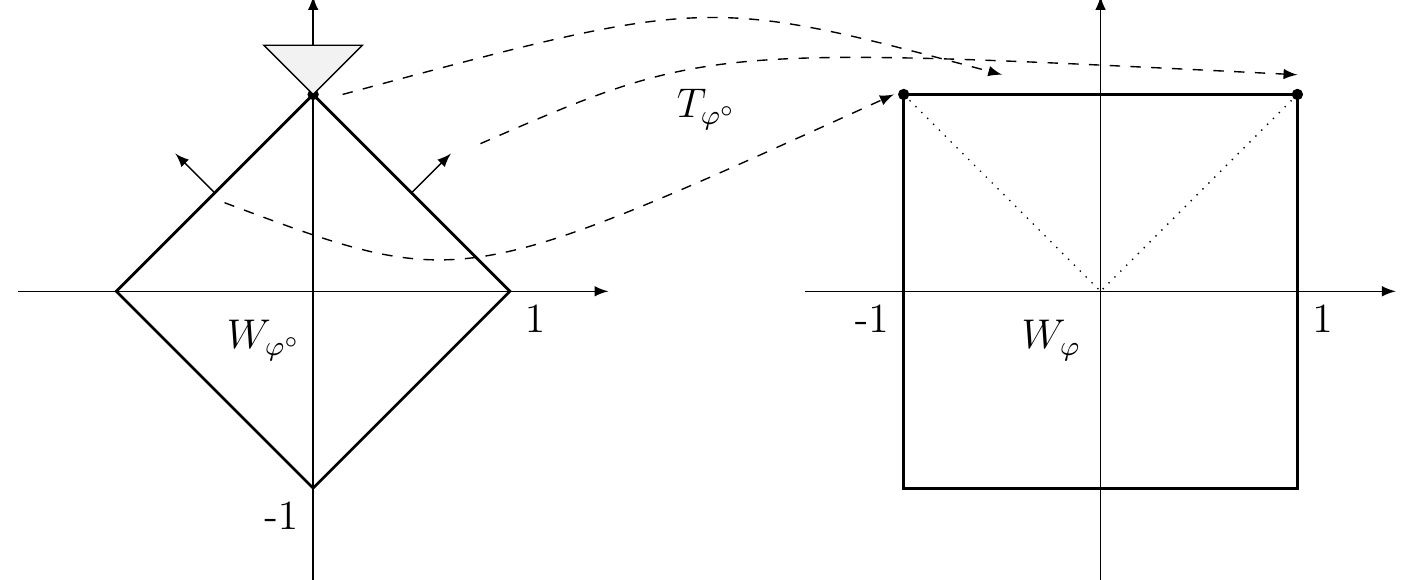}
\caption{The square crystalline norm $\varphi$ and the map 
$T_{\varphi^\circ}$}\label{fig:fig5}
\end{figure}

If $\varphi^\circ(\xi^\circ)=1$, a direct computation gives
\begin{equation}\label{e:dualcris}
T_{\varphi^\circ}(\xi^\circ)  =
\begin{cases}
(1,1) & \xi^\circ\in\oray{e_1,e_2} \\
(1,-1) & \xi^\circ\in\oray{e_1,-e_2}\\
(-1,-1) & \xi^\circ\in\oray{e_2,-e_1}\\
(-1,1) & \xi^\circ\in\oray{-e_2,-e_1} 
\end{cases} 
\qquad 
\begin{cases}
T_{\varphi^\circ}(e_1) & =\cray{(1,1),(1,-1)}, \\
T_{\varphi^\circ}(e_2) & =\cray{(-1,1),(1,1)},\\ 
T_{\varphi^\circ}(-e_1) & =\cray{(-1,1),(-1,-1)},\\ 
T_{\varphi^\circ}(-e_2) & =\cray{(-1,-1),(1,-1)}.
\end{cases}
\end{equation}
(see Figure \ref{fig:fig5}).

The notion of intrinsic curvature in $(\R^2, \varphi)$ is based on the 
existence of regular selections of 
$T_{\phi^\circ}\bigl(\nabla\dist^E\bigr)$ on $\partial E$.

\begin{Definition}[$\varphi$--regular set, Cahn--Hoffmann field, mean 
$\varphi$--curvature]\label{d:ch}
We say that an open set $E\subseteq \R^2$ is {\em  $\varphi$--regular} if 
$\partial E$ is 
a compact Lipschitz curve, and there exists a 
vector field $n_\varphi\in \Lip(\partial E;\R^2)$ such that 
$n_\varphi \in T_{\varphi^\circ}(\nabla \dist^E)$ $\haus$--almost everywhere in 
$\partial E$.

\noindent Any such a selection of the multivalued 
function $T_{\varphi^\circ}(\nabla \dist^E)$
on $\partial E$ is called  a {\em  Cahn--Hoffmann vector field} for $\partial 
E$  
associated to 
$\varphi$, and $\curf=\dive n_\varphi$ is the related {\em  mean 
$\varphi$--curvature (crystalline square mean curvature) of $\partial E$}.
\end{Definition}

\begin{rk} Notice that, unlike the outer (euclidean) unit normal $\ne$, a 
Cahn--Hoffmann field is 
defined everywhere on $\partial E$. 
\end{rk}

\begin{rk}
Any Cahn--Hoffmann vector field $n_\varphi$ has 
gradient orthogonal to the normal direction, so that $\dive n_\varphi$ equals 
the 
tangential divergence of the field.
\end{rk}

\begin{figure}[h!]
\includegraphics[height=3cm]{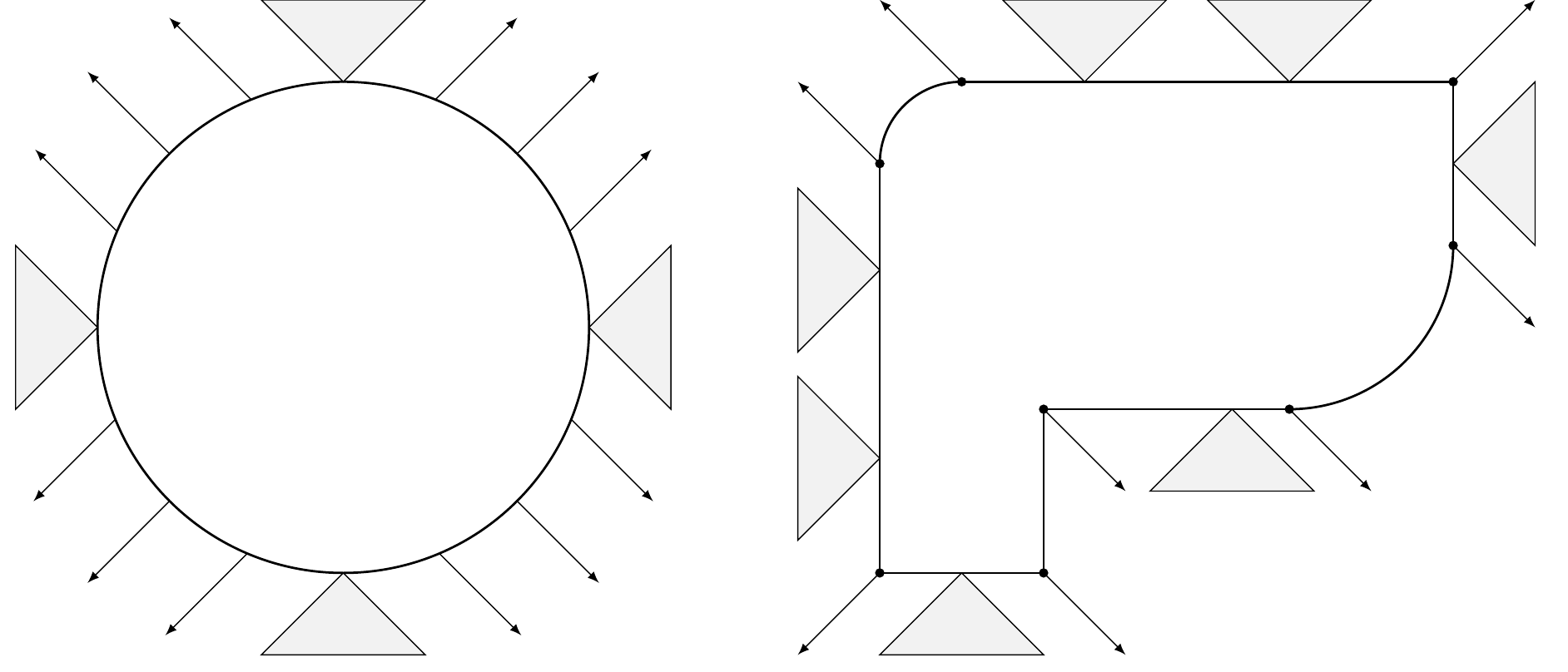}
\caption{A non $\varphi$--regular set (left) and a $\varphi$--regular set 
(right) }\label{fig:fig6}
\end{figure}

\begin{rk}[Edges and vertices]
The boundary of a $\varphi$--regular set $E$ is given by a finite number
of closed arcs with the property that $T_{\varphi^\circ}(\nabla \dist^E)$
is a fixed set $T_A$ in the interior of each arc $A$ (see Figure 
\ref{fig:fig6}). This set 
$T_A$ 
is either a singleton,  if the arc $A$ is not a horizontal or vertical 
segment, or one of the closed convex cones described in \eqref{e:dualcris}. 
The arcs of $\partial E$ which are straight horizontal or 
vertical segments 
will be called {\em  edges}, and the endpoints of every arc will be called 
{\em  vertices} of $\partial E$.
\end{rk}

\begin{rk}\label{r:nvertex}
The requirement of Lipschitz continuity keeps the  value of every 
Cahn--Hoffmann vector field fixed at vertices.
Hence, in 
order to exhibit a Cahn--Hoffmann vector field $n_\varphi$ on $\partial E$ it 
is enough to construct a field $n_A\in \Lip(A;\R^2)$ on each arc $A$, with the 
correct values 
at the vertices, and satisfying the constraint $n_A 
\in T_A$.
In what follows, with a little abuse of notation, we shall call $n_A$ the 
Cahn--Hoffmann vector field on the arc $A$.
\end{rk}

\subsection{Forced crystalline curvature flow}

The  appropriate way for describing a geometric evolution $E(t)$,
starting from a given $\varphi$--regular set $E$, and
trying to reduce as fast as 
possible the energy functional 
\[
F(E(t))=P_\varphi(E(t))+V(E(t))=\int_{\partial E(t)}\bigl( 
|\nu_1^E(t)|+|\nu_2^E(t)|\bigr)\, d\haus +\int_{E(t)} 
f\ 
d\leb,
\] 
is a suitable weak version of the law
\[
v= \kappa_\varphi+f, \qquad \text{on}\ E(t)
\]   
where $v$ is the scalar velocity of $\partial E(t)$ along 
the normal direction.

\begin{Definition}
Given $T>0$, we say that a family $E(t)$, $t\in [0,T]$, is a  
{\em  crystalline mean 
curvature flow 
in $[0,T)$  with forcing term $f\in L^{\infty}(\R^2)$}
if 
\begin{itemize}
\item[(i)] $E(t)\subseteq\R^2$ is a Lipschitz set for every $t\in [0,T)$;
\item[(ii)] there exists an open set $A\subseteq \R^2\times [0,T)$ such that
$\bigcup_{t\in [0,T)}\partial E(t)\times \{t\}\subseteq A$, and the function
$d(\xi,t)\doteq d^{E(t)}_\varphi(\xi)$ is locally Lipschitz in $A$;
\item[(iii)] there exists a function $n \in L^\infty(A,\R^2)$, with  
$\dive n\in L^\infty(A)$, such that the restriction of  $n(t, \cdot)$ to
$\partial E(t)$ is a 
Cahn--Hoffmann vector field 
for $\partial E(t)$ for every $t\in [0,T]$;
\item[(iv)] $\partial_t d-\dive n \in [f^-, f^+]$ 
$\haus$--almost everywhere in $\partial E(t)$ and for all $t\in[0,T)$, where
\[
f^+(\xi)=\text{ess} \limsup_{\eta\to \xi}f(\eta), \qquad 
f^-(\xi)=\text{ess} \liminf_{\eta\to \xi}f(\eta), \qquad \xi\in\R^2.
\]
\end{itemize}
\end{Definition}

As underlined in Remark \ref{r:nvertex}, 
the Cahn--Hoffmann vector field $n(t)$ on $\partial E(t)$ is obtained by
gluing suitable fields constructed on each arc of the boundary.
If an arc $A\subseteq \partial E(t)$ is not a edge,  the  
Cahn--Hoffmann vector field $n(t)$ is uniquely determined by the geometry 
of 
$A$, and the $\varphi$--mean 
curvatures of at each point of $A$ is zero. 
On the other hand, there are infinitely many choices for the vector 
field in the interior of the  
edges on $\partial E(t)$, and hence infinitely many 
different $\varphi$--mean curvatures associated.
Most of these choices are meaningless for the point of view  of the geometric 
evolution.
In order to overcome this 
ambiguity we fix the choice by a variational selection principle which turns 
out to be consistent with the curve shortening flow (see 
\cite{BNP1}, \cite{BNP2}, \cite{BNP3}, \cite{GR1}, \cite{GR2}).

\begin{Definition}[Variational forced crystalline curvature flow]
A {\em  variational forced crystalline curvature flow} is a forced crystalline 
curvature flow
$E(t)$, $t\in [0,T)$, such that for every $t\in [0,T)$ and for every edge
$L$ of $\partial E(t)$ the Cahn--Hoffmann vector field $n_{L}$ is 
the unique minimum of the functional
\[
\mathcal{N}_{L}(n)= \int_{L} |f-\dive n|^2\, d\haus
\] 
in the set
\[
D_L=\left\{ 
n \in L^\infty(L,\R^2), n\in T_{L}, \dive n\in L^\infty(L),
n(p)=n_0, n(q)=n_1
\right\}
\]
where $p$, $q$ are the endpoints of $L$ and $n_0$, $n_1$ are the 
values at the vertices $p$, $q$ assigned to every Cahn--Hoffmann vector field
(see Remark \ref{r:nvertex}). 
\end{Definition}

\begin{rk}\label{r:calsets}
If the minimum $n_{L}$ in $D_L$ of the functional $\mathcal{N}_{L}$ 
satisfies the strict constraint $n_{L}(\xi)\in \inte T_{L}$ for every 
$\xi\in L$, then the velocity $f-\dive n_{L}$ is constant along the edge, 
that is the flat arc remains flat under the evolution. This is always the case
when $f=0$, since the unique minimum is the
interpolation of the assigned values at the vertices of $L$, and  the constant 
value 
of the mean $\varphi$--curvature is given by
\begin{equation}\label{f:flatc}
\kappa_\varphi^{L}=\chi_{L}\frac{2}{\ell} \ \text{on}\ L,
\end{equation} 
where $\ell$ is the length of the edge $L$ and $\chi_{L}$ is a 
convexity factor:
$\chi_{L}=1,-1,0$, depending on whether $E(t)$ is locally convex at $L$, 
locally 
concave at $L$, or neither.
\end{rk}

We refer to \cite{CN,GGR,BGNo} for some existence and uniqueness 
results for variational forced crystalline curvature flow,  when the forcing 
term $f$ is a Lipschitz function. To the best of our knowledge, there 
is no general results for discontinuous forcing terms. 

The easiest example of variational forced crystalline curvature flow is the one
starting from a coordinate rectangle $R$ (rectangle with 
edges parallel to the coordinate axes), and with constant forcing term 
$f(\xi)=\gamma\in \R$. 

Ordering the vertices of $\partial E$ clockwise starting from the left--upper 
corner, $P_i$, $i=1,\dots,4$, we have 
\[
n_\varphi(P_1)=(-1,1), \quad n_\varphi(P_2)=(1,1), \quad n_\varphi(P_3)=(1,-1), 
\quad
n_\varphi(P_4)=(-1,-1),
\]
and hence the variational Cahn--Hoffmann vector field on the sides is
\[
n_\varphi(\xi)=n_\varphi(P_i)+ \frac{2(-1)^{i+1}}{\ell_i}(\xi-P_i), \qquad 
\xi\in 
L_i,\ 
i=1, \dots, 4,
\]
where $L_i$ is the edge of the rectangle starting from $P_i$, and $\ell_i$ is 
its 
length.

The $\varphi$--curvature  associated to this field is constant on each 
edge $L_i$, and  
\[
\kappa_\varphi^i=\frac{2(-1)^{i+1}}{\ell_i} \ \text{on}\ L_i, \ i=1, \dots, 4.
\]
As a consequence the evolution of $R$ is given by rectangles 
$R(t)$, and the 
description of the flow reduces to the analysis of a system of ODE's
solved by the 
length $\ell_1(t)$ and $\ell_2(t)$ of the horizontal and the vertical edges of 
$R(t)$: 
\begin{equation}\label{f:forzc}
\begin{cases}
\ell_1'=-\dfrac{4}{l_2}-2\gamma \,, \\[8pt]
\ell_2'=-\dfrac{4}{l_1}-2\gamma \,.
\end{cases}
\end{equation}
(see \cite{AT} and \cite{GG}).
Hence, if $\gamma>0$, $\ell_1$ and $\ell_2$ are decreasing and we have 
finite--time extinction (see the phase portrait in the left--hand side 
of 
Figure \ref{fig:pp}), while if $\gamma<0$ the evolution is the one depicted 
in the right--hand side of Figure \ref{fig:pp}. The 
square of side length $\ell^0=-2/\gamma$ is the unique equilibrium of the 
system. 
Moreover, the function
\[
U(\ell_1,\ell_2)=4(\log(\ell_2)-\log(\ell_1))+2\gamma(\ell_2-\ell_1) 
\]
is a constant of motion for the system \eqref{f:forzc}.
The squares starting with a side length 
shorter than $l^0$ shrink to a point, while the squares starting with a side 
length longer than $l^0$ expand with asymptotic velocity $-\gamma$ on each side 
as the side length diverges. The rectangles can shrink to a point, 
converge to the equilibrium or expand, depending on the starting 
length of the edges.

\begin{figure}[h!]
\includegraphics[height=4cm]{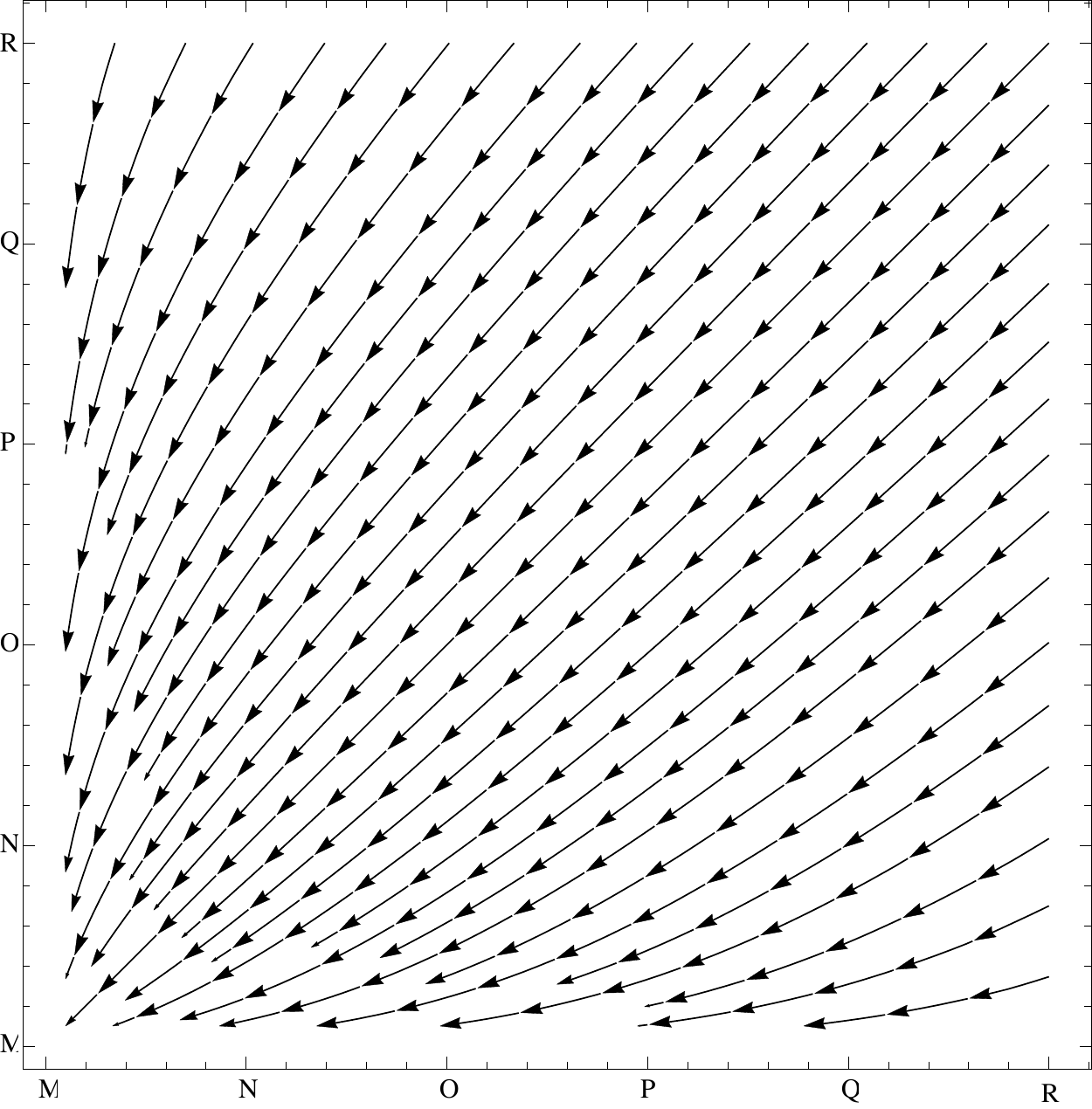}
\qquad \qquad
\includegraphics[height=4cm]{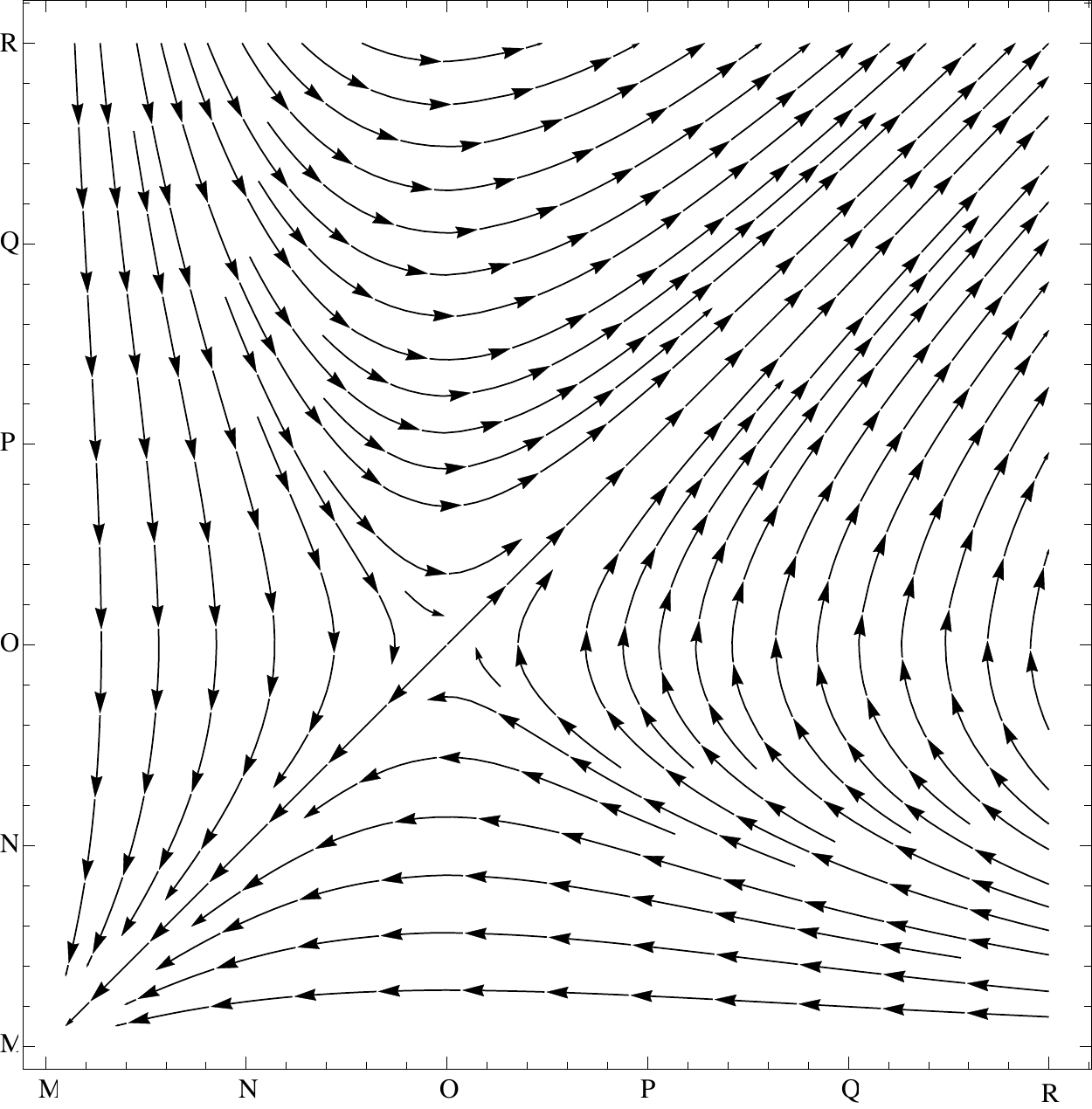}
\caption{Phase portraits of \eqref{f:forzc} for $\gamma>0$ and 
for $\gamma<0$. }\label{fig:pp}
\end{figure} 

In general, in presence of a spatially inhomogeneous forcing term,  the 
selfsimilarity of the evolution of 
coordinate rectangles may fail, due to the possibility of edge 
breaking/bending phenomena (see \cite{GR1}, \cite{GR2}).

\subsection{The oscillating forcing term}

In what follows we will consider a prototypical case of oscillating layered  
forcing term: given two constants $\alpha< 0 <\beta$, let
\begin{equation}\label{f:defg}
g(x)=
\begin{cases}
\alpha, & \text{if}\ \text{dist}(x,\Z)\leq \dfrac{1}{4},\\
\beta, & \text{otherwise},
\end{cases}
\end{equation}
and let us denote by $\mathcal{I}$ the set of discontinuity lines 
of the function $g$:
\[
\mathcal{I}:=\left\{(x,y),\ 
x\in\dfrac{1}{4}+\dfrac{1}{2}\Z, \ y\in\R\right\}.
\]
In order to distinguish the two families of discontinuity lines for 
$g$, depending on the position of the 
phases $\alpha$ and $\beta$ with respect to the interface, we will use the 
notation 
\[
\begin{split}
\mathcal{I}_{\alpha,\beta} & :=\left\{(x,y),\ y\in\R,\ x\in\R\  \text{such 
that}\  
g(s,y)=\beta\ \forall s\in (x, x+1/2)\right\} \\ 
\mathcal{I}_{\beta,\alpha} & :=\left\{(x,y),\ y\in\R,\ x\in\R\  \text{such 
that}\  
g(s,y)=\alpha\ \forall s\in (x, x+1/2)\right\}.
\end{split}
\]

Given $\varepsilon >0$ we consider the function $g_\varepsilon$ defined by
\[
g_\varepsilon(x,y):=g\left(\frac{x}{\varepsilon}\right). 
\]
(with an abuse of notation, we will often write $g_\varepsilon(x)$ 
instead of $g_\varepsilon(x,y)$).
Setting 
\begin{equation}\label{f:xenne}
x_N:=\left(N+\frac{1}{4}\right)\varepsilon,\qquad N\in\N,
\end{equation} 

\begin{figure}[h!]
\includegraphics[height=2.5cm]{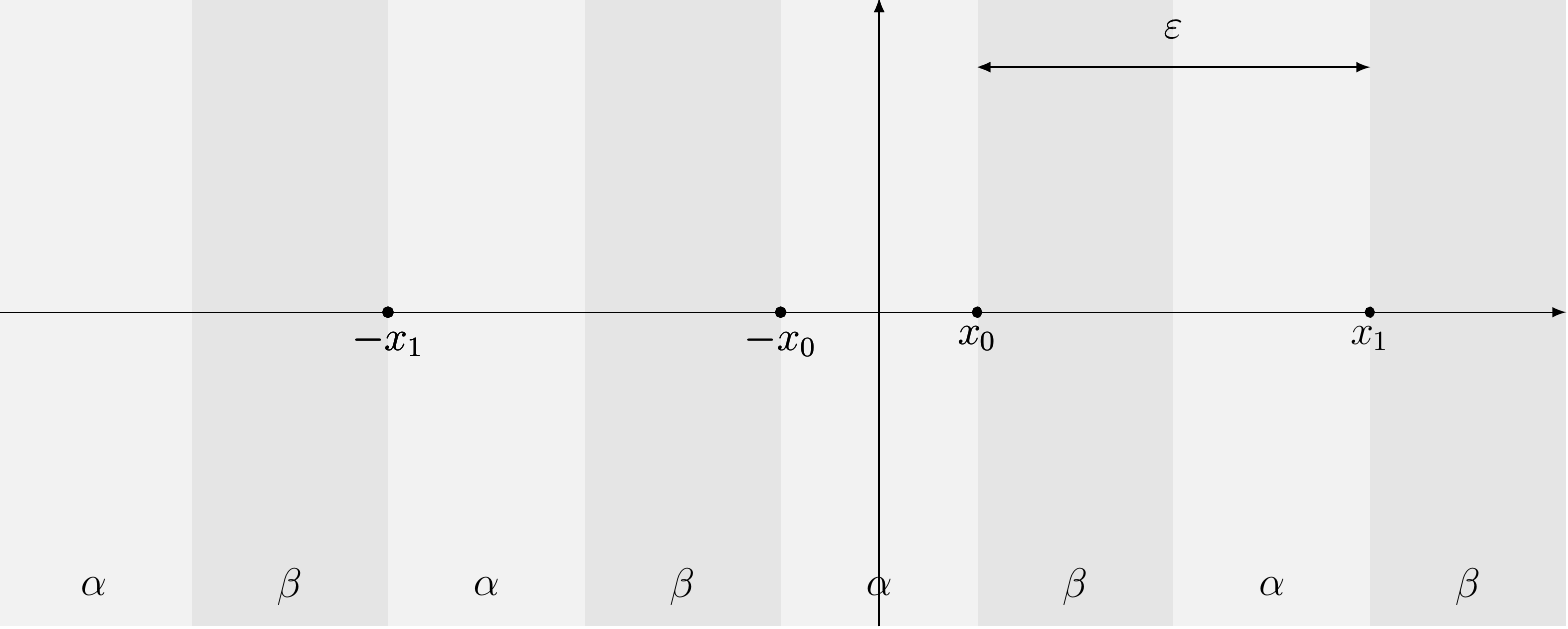}
\caption{The oscillating forcing term $g_\varepsilon$}\label{fig:fig7}
\end{figure}
we have that
\begin{equation}
g_\varepsilon(x,y)=
\begin{cases}
\alpha, & x\in \left(x_N-\dfrac{\varepsilon}{2}, x_N\right),\ y\in\R, \\[10pt]
\beta, & x\in \left(x_N, x_N+\dfrac{\varepsilon}{2}\right), \ y\in\R,
\end{cases}
\end{equation}
that is $\{x=x_N\}\subseteq \varepsilon\mathcal{I}_{\alpha,\beta}$ and 
$\{x=-x_N\}\subseteq \varepsilon\mathcal{I}_{\beta, \alpha}$ for every $N\in\N$ 
(see 
Figure \ref{fig:fig7}).

Finally, we define the multifunction
\[
G_\varepsilon (x,y)=
\begin{cases}
g_\varepsilon(x,y), & \text{if}\ (x,y)\not\in \varepsilon\mathcal{I} , \\
[\alpha, \beta], &  \text{if}\ (x,y)\in \varepsilon\mathcal{I},
\end{cases}
\] 
in such a way a variational crystalline mean curvature flow 
$E(t)$ with forcing term 
$g_\varepsilon$ has to have normal velocity
$v(t)\in \dive n(t) + G_\varepsilon$ on $\partial E(t)$.

\section{Calibrable edges}\label{s:calibrability}

In this section we keep $\varepsilon>0$ fixed, and we focus our 
attention on 
the effect of the forcing term $g_\varepsilon$ along a horizontal or 
vertical edge $L$. 



\begin{Definition}
An edge $L$ of a $\varphi$--regular set $E$ is {\em  calibrable} if there 
exists a 
Cahn--Hoffmann vector field $n$ for $\partial E$, and a constant $v\in \R$ such 
that
$v \in \dive n +G_\varepsilon$ on $L$. In this case we say that $v$ is 
a  {\em 
(normal) velocity 
of the 
edge $L$}.
\end{Definition}

\begin{rk}
In view of Remark \ref{r:nvertex}, in order to show that an edge $L$ of 
$\partial E$ is calibrable it is enough to show that there exists a vector 
field $n$ on $L$, such that $n(\xi)\in T_L$ for every $\xi\in \inte L$ and 
agrees with assigned values (prescribed by the 
geometry of $\partial E$) at the endpoints of $L$, .
\end{rk}

In what follows, the length of an edge $L$ will be denoted by $\ell$.

\subsection{Vertical edges}\label{rk:vert}
Every vertical edge is calibrable. 
Namely, if $x=\overline{x}\in\R$ is the straight line containing $L$, there 
exists a constant selection 
$\gamma_\varepsilon(\overline{x})$ of $G_\varepsilon$ on $L$, and hence 
the Cahn-Hoffman field given by the linear interpolation of the extreme 
values
satisfies all the 
requirements. The related velocity of the edge is given by
\begin{equation}\label{f:vveloc}
v= 
\frac{2}{\ell}\chi_{L}+\gamma_\varepsilon(\overline{x}),
\end{equation}
where $\chi_{L}$ is the convexity factor defined in Remark 
\ref{r:calsets}.
Hence the velocity of the edge is uniquely determined if $L$ is not a subset of the jump 
set $\varepsilon\mathcal{I}$  of 
$g_\varepsilon$, while, for 
$\overline{x}$ belonging to an interface we 
can freely choose any fixed 
value $v$ such that
\[
v\in 
\left[\frac{2}{\ell}\chi_{L}+\alpha, 
\frac{2}{\ell}\chi_{L}+\beta\right].
\]
In particular, a vertical edge $L\subseteq\varepsilon\mathcal{I}$ with 
zero $\varphi$--curvature is allowed to have 
velocity zero, 
since 
we can choose $\gamma_\varepsilon=0$ on $L$. Similarly, a vertical edge 
$L\subseteq\varepsilon\mathcal{I}$ either with 
positive $\varphi$--curvature 
and length $\ell\geq 
-2/\alpha$, or  with negative $\varphi$--curvature and length $\ell\geq 
2/\beta$ is allowed to have velocity zero, since 
we can choose $\gamma_\varepsilon=-2/\ell$ or $\gamma_\varepsilon=2/\ell$ on 
$L$, 
respectively.

\subsection{Horizontal edges}

Let $L$ be a horizontal edge.
A Cahn--Hoffman vector field $n_\varphi$ on $L$ belongs to
$T_{\varphi^\circ}(\nabla \dist^E)=T_{\varphi^\circ}(\pm e_2)$, 
so that its second component is fixed (see \eqref{e:dualcris}).
In what follows we consider
only its first component, by using the abuse of notation 
$n_\varphi=(n(x),\pm 1)$ on $L$.

Hence $L=[p,q]\times \{\overline{y}\}$ turns out to be calibrable 
if there exists a Lipschitz function 
\begin{equation}\label{f:constr}
n\colon [p,q]\to [-1,1],
\end{equation}
such that 
\begin{equation}\label{f:nofrac2}
n'+g_\varepsilon= \chi_L\frac{2}{\ell}+ \frac{1}{\ell} \int_L 
g\left(\dfrac{s}{\varepsilon}\right) \, ds \qquad \text{a.e.\ in}\ 
[p,q],
\end{equation}  
and with the following prescribed values at the endpoints
\begin{equation}\label{e:bcon}
(BC)=\begin{cases}
n(p)=n(q)=n_0\in \{\pm 1\} &  \text{if}\ \chi_L=0; \\
n(p)=-1,\ n(q)=1, &  \text{if}\ \chi_L=1; \\
n(p)=1,\ n(q)=-1, &  \text{if}\ \chi_L=-1.
\end{cases}
\end{equation}

Denoting by $\ell_{\alpha}$, $\ell_{\beta}\in [0,\varepsilon/2]$ the 
non--negative lengths given by the conditions 
\begin{equation}\label{e:pezzetti}
 \ell - 
 \varepsilon\left\lfloor\dfrac{\ell}{\varepsilon}\right\rfloor=\ell_{\alpha}+ 
 \ell_{\beta}, \qquad 
 \int_L g_\varepsilon(s)\, ds= \frac{\alpha+\beta}{2}\left(\ell-
 \ell_{\alpha}-\ell_{\beta}\right)+\alpha \ell_{\alpha}+\beta \ell_{\beta},
\end{equation}
the necessary condition \eqref{f:nofrac2} prescribes the value of $n'$ outside 
the jump set of $g_\varepsilon$:
\begin{equation}\label{f:conden3}
n'(x) = 
 \begin{cases}
\dfrac{1}{2\ell}\left(4 
\chi_L+(\beta-\alpha)(\ell-\ell_{\alpha}+\ell_{\beta})\right)
&  \text{if}\ 
 g_\varepsilon(x)=\alpha , \\[10pt]
 \dfrac{1}{2\ell}\left(4 
 \chi_L-(\beta-\alpha)(\ell+\ell_{\alpha}-\ell_{\beta})\right),
  & \text{if}\ 
 g_\varepsilon(x)=\beta.
 \end{cases}
\end{equation}
and  the velocity of the edge $L$:
\begin{equation} \label{f:velocita}
v_L=\chi_L\frac{2}{\ell}+\frac{\alpha+\beta}{2}+
\frac{\beta-\alpha}{2\ell}(\ell_\beta -\ell_\alpha).
\end{equation} 

In conclusion, the calibrability conditions \eqref{f:nofrac2} and 
\eqref{e:bcon} determine univocally a
candidate field $n$ (and the related velocity of the edge), which is continuous 
and affine
with given slope in each phase of $g_\varepsilon$. This field $n$
is the Cahn--Hoffman field which calibrates $L$ with velocity 
\eqref{f:velocita} if and only if it satisfies the constraint 
$|n(x)| \leq 1$ for every $x\in [p,q]$.

\begin{rk}\label{r:incrper}
If $\ell> \varepsilon$, by \eqref{f:conden3} the variation of $n$ in a 
period 
$\varepsilon$ is
\begin{equation}\label{e:inaper}
\Delta_\varepsilon n= n(x+\varepsilon)-n(x)= \frac{\varepsilon}{2 \ell}
\left(4 \chi_L+(\beta-\alpha)(\ell_\beta-\ell_\alpha)\right).
\end{equation}
In what follows we will assume 
\begin{equation}
\label{f:asseps}
0<\varepsilon < \frac{8}{\beta-\alpha}.
\end{equation} 
In particular, \eqref{f:asseps} implies that  $\Delta_\varepsilon n$ has the 
same sign of $\chi_L$, and hence  the constraint $|n|\leq 1$ is fulfilled in 
$[p,q]$ if it is satisfied in a suitable neighborhood of the extreme points 
$p$ and $q$.
\end{rk}

The calibrability of a horizontal edge $L$ will depend   
on its length and on the position of its endpoints. We start by characterizing 
the calibrable horizontal edges with zero $\varphi$--curvature.
 
\begin{prop}[Horizontal edges with zero 
$\varphi$--curvature]\label{r:hcurvz}
Let $L=[p,q]\times \{\overline{y}\}$ be a horizontal edge  with zero 
$\varphi$--curvature, let $\ell_{\alpha}$, $\ell_{\beta}$ be the lengths 
defined in \eqref{e:pezzetti}, and  let 
$n_0\in \{\pm 1\}$ be the given value of the 
Cahn--Hoffmann vector field at the endpoints of $L$.
Then the following 
hold.
\begin{itemize}
\item[(i)] If $\ell=\ell_\alpha+\ell_\beta < \varepsilon$, 
$L$ is calibrable with velocity 
$
v_L=\dfrac{\alpha \ell_{\alpha}+\beta 
\ell_{\beta}}{\ell_\alpha+\ell_\beta}
$
if and only if
\begin{itemize}
\item[(ia)] $n_0=1$, and either $g_\varepsilon(p)=\beta$, 
$g_\varepsilon(q)=\alpha$, or with an endpoint on 
$\varepsilon\mathcal{I}_{\alpha, 
\beta}$  ;
\item[(ib)] $n_0=-1$, and either $g_\varepsilon(p)=\alpha$, 
$g_\varepsilon(q)=\beta$, or with an endpoint on 
$\varepsilon\mathcal{I}_{\beta,\alpha}$ .
\end{itemize}
\item[(ii)] If $\ell\geq  \varepsilon$, $L$ is calibrable 
with velocity 
$
v_L=\dfrac{\alpha+\beta}{2}
$
if and only if
\begin{itemize}
\item[(iia)] $n_0=1$, and  $(p,\overline{y})$, $(q,\overline{y}) \in 
 \varepsilon\mathcal{I}_{\alpha, \beta}$;
\item[(iib)] $n_0=-1$, and $(p,\overline{y})$, $(q,\overline{y}) \in 
  \varepsilon\mathcal{I}_{\beta,\alpha}$.
\end{itemize}
\end{itemize}

\end{prop}
 
 \begin{proof}
We prove the case $n_0=1$, the other one being similar.
By \eqref{f:conden3}, the unique candidate field $n$ is strictly monotone 
increasing in every $\alpha$ phase, and 
strictly monotone decreasing in every $\beta$ phase, and, in order 
to satisfy the constraint $|n|\leq 1$ on $L$, the edge needs 
to be the union of three consecutive segments
$L= L_\beta \cup L_c \cup  L_\alpha$, with 
$L_\beta=[p,p+\ell_\beta]\times \{\overline{y}\}$, 
$L_c=[p+\ell_\beta, q-\ell_\alpha]\times \{\overline{y}\}$,
$L_\alpha=[q-\ell_\alpha,q]\times \{\overline{y}\}$, with
$p+\ell_\beta$, $q-\ell_\alpha \in \varepsilon\mathcal{I}_{ 
\beta,\alpha}$.
If $L_c=\emptyset$, then the constraint $|n|\leq 1$ is satisfied on $L$
if and only if 
\[
-\ell_\beta 
\frac{\beta-\alpha}{2(\ell_\alpha+\ell_\beta)}(\ell+\ell_\alpha-\ell_\beta)=
-\ell_\beta\ell_\alpha\frac{\beta-\alpha}{(\ell_\alpha+\ell_\beta)} \geq -2.
\] 
Under the assumption \eqref{f:asseps}
this is always the case, since
\[
(\beta-\alpha) \ell_\alpha \ell_\beta \leq \frac{8}{\varepsilon} \ell_\alpha 
\ell_\beta \leq 4 \min\{\ell_\alpha, \ell_\beta\} \leq 2(\ell_\alpha + 
\ell_\beta),
\]
proving (i).
 
On the other hand, if $L_c\neq \emptyset$, 
by \eqref{e:inaper} we have 
\[
n(p+\varepsilon)-n(p)=\frac{\varepsilon}{2\ell}(\beta-\alpha)
(\ell_\beta-\ell_\alpha)=n(q)-n(q-\varepsilon),
\]
and hence,
since $n(p)=n(q)=1$, the constraint $|n|\leq 1$ is not satisfied if
$\ell_\alpha \neq \ell_\beta$. Finally,
if  $\ell_\alpha=\ell_\beta$, then
\[
  n'(x)= 
  \begin{cases}
  \dfrac{\beta-\alpha}{2}
  &  \text{if}\ 
  g_\varepsilon(x)=\alpha , \\[10pt]
  \dfrac{\alpha-\beta}{2}
   & \text{if}\ 
  g_\varepsilon(x)=\beta,
  \end{cases}
\]
and a Canh--Hoffmann vector field with this derivative exists
only if  $\ell_\alpha=\ell_\beta=\varepsilon/2$,
otherwhise
\[
 n(p+\ell_\beta+\varepsilon/2)= 1+
 \frac{\beta-\alpha}{2}\left(\frac{\varepsilon}{2}-\ell_\beta\right)>1.
\]
In conclusion, $L$ is calibrable with velocity $v_L=(\alpha+\beta)/2$
if and only if $(p,\overline{y})$, $(q,\overline{y}) \in 
 \varepsilon\mathcal{I}_{\alpha, \beta}$.
 \end{proof}

\begin{rk}\label{e:latino}
If $L=[x_N, x_N+\delta] \times \{\overline{y}\}$, with $\overline{y}\in\R$, 
$\delta\in (0, \varepsilon)$, and  $x_N$ defined in 
\eqref{f:xenne},
is a horizontal edge with zero $\varphi$--curvature (see Figure 
\ref{fig:fig9}, right), then  $L$ is 
calibrable by Proposition \ref{r:hcurvz}(i). 
More precisely, if $\delta \leq \varepsilon/2$, 
then $g_\varepsilon=\beta$ on $L$, and we can take 
$n$ constant on $L$, so that $L$ has constant velocity $v_L=\beta$. 
On the 
other hand, if $\delta>\varepsilon/2$, then the field
\begin{equation}\label{f:chlatino}
n(x)=n(x_N) 
+\frac{\frac{\varepsilon}{2}\beta+\left(\delta-\frac{\varepsilon}{2}\right) 
\alpha}{\delta} (x-x_N)-\int_{x_N}^x \geps\left(\frac{s}{\varepsilon}\right)\, 
ds
\end{equation}
calibrates the edge $L$ with velocity
\begin{equation}\label{f:vldelta}
v_{L}= 
\frac{\frac{\varepsilon}{2}\beta+\left(\delta-\frac{\varepsilon}{2}\right)
 \alpha}{\delta}, \qquad \delta \in (\varepsilon/2, \varepsilon).
\end{equation}
\end{rk}
    
Concerning the edges with non zero $\varphi$--curvature, the following result
shows that the edge is always calibrable when the curvature term is 
dominant.

\begin{prop}\label{p:fracon}
Every horizontal edge $L$ such that
\begin{itemize}
\item[($C_+$):] $\chi_L=1$, and $\ell+\ell_\alpha-\ell_\beta \leq 
4/(\beta-\alpha)$;
\item[($C_-$):] $\chi_L=-1$, and $\ell-\ell_\alpha+\ell_\beta \leq 
4/(\beta-\alpha)$;
\end{itemize}
is calibrable with velocity $v_L$ given by \eqref{f:velocita}.
\end{prop}
\begin{proof}
By \eqref{f:conden3}, in both cases the candidate field $n$ varies  
monotonically
between the two extreme values: in the case ($C_+$) it is a increasing function
 from $-1$ to $1$, while in the case ($C_-$) it is a decreasing function
from $1$ to $-1$. This ensures that the constraint \eqref{f:constr} 
is satisfied.
\end{proof}

When the forcing term dominates the curvature term, the calibrability may fail 
(see Proposition \ref{p:gencalibr} below). 
Nevertheless, the edges with endpoints
on suitable interfaces are always calibrable, as we show  
in the following result.

\begin{prop}\label{p:oninterf}
Let $L=[p,q]\times \{\overline{y}\}$ be a horizontal edge with 
$\ell\geq\varepsilon$. Then the following 
hold.
\begin{itemize}
\item[(i)] If 
$\chi_L=1$, $p\in \varepsilon\mathcal{I}_{\beta,\alpha}$, 
$q \in \varepsilon\mathcal{I}_{\alpha,\beta}$, then $L$ 
is calibrable
with velocity
\[
v_L= \frac{2}{\ell} +\frac{\alpha+\beta}{2}- 
\frac{(\beta-\alpha)\varepsilon}{4 \ell}.
\]
\item[(ii)] If 
$\chi_L=-1$, $p\in \varepsilon\mathcal{I}_{\alpha,\beta}$, 
$q \in \varepsilon\mathcal{I}_{\beta, \alpha}$, then $L$ 
is calibrable
with velocity
\[
v_L= -\frac{2}{\ell} +\frac{\alpha+\beta}{2}+ 
\frac{(\beta-\alpha)\varepsilon}{4 \ell}.
\]
\end{itemize}
\end{prop}
\begin{proof}
Assume that $\chi_L=1$, $p\in \varepsilon\mathcal{I}_{\beta,\alpha}$, 
$q \in \varepsilon\mathcal{I}_{\alpha,\beta}$, so that $n(p)=-1$, $n(q)=1$,
$\ell_\alpha=\varepsilon/2$, and $\ell_\beta=0$. Then, by \eqref{f:conden3}, we 
have that the candidate Cahn-Hoffmann field $n$ is increasing in 
$[p,p+\varepsilon/2]$, and, by \eqref{e:inaper} and \ref{f:asseps}, 
\[
n(p+\varepsilon)-n(p)=\frac{\varepsilon}{4\ell} 
\left(8-(\beta-\alpha)\varepsilon\right) >0.
\]
Similarly, we have that $n$ satisfies the constraint also in 
$[q-\varepsilon,q]$, and, again by \eqref{e:inaper}, we conclude that
$|n|\leq 1$ on $L$ so that $L$ is calibrable 
with velocity $v_L$ 
given by (i). The proof of (ii) is similar. 
\end{proof}

\begin{figure}[h!]
\includegraphics[width=4.8cm]{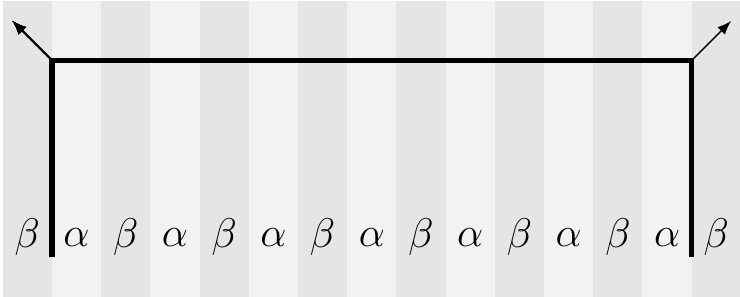}
\quad
\includegraphics[width=4.8cm]{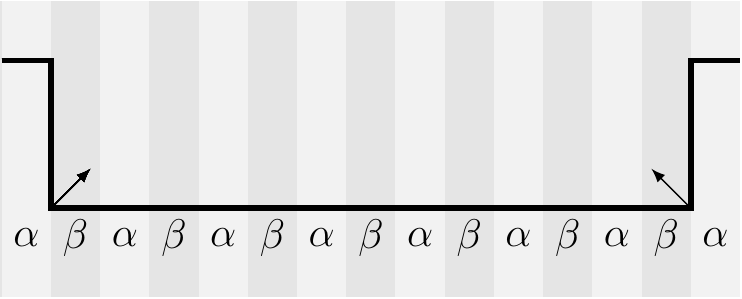}
\quad
\includegraphics[width=3.2cm]{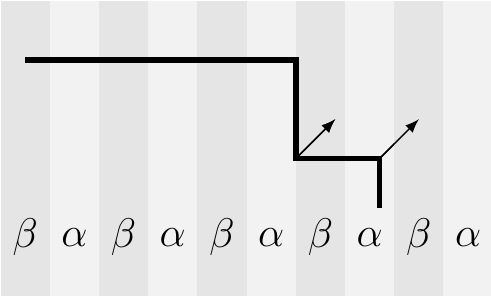}
\caption{The calibrable edges in Propositions \ref{p:oninterf}(i), 
\ref{p:oninterf}(ii),
and \ref{r:hcurvz}(ia)}\label{fig:fig9}
\end{figure}

\begin{Definition}[$\mathcal{C}$--edges]\label{d:cedge}
We say that a horizontal edge $L=[p,q]\times \{\overline{y}\}$ is a
\textit{$\mathcal{C}$--edge} if is of one of the following types. 
\begin{itemize}
\item [($\mathcal{C}^+)$] $\chi_L=1,\ (p,\overline{y})\in 
\varepsilon\mathcal{I}_{\beta,\alpha},\ (q,\overline{y}) \in 
\varepsilon\mathcal{I}_{\alpha,\beta}$.
\item [($\mathcal{C}^-$)]$\chi_L=-1,\  (p,\overline{y})\in 
\varepsilon\mathcal{I}_{\alpha,\beta},  (q,\overline{y}) \in 
\varepsilon\mathcal{I}_{\beta, 
\alpha}$.
\item [($\mathcal{C}^0$)] $\chi_L=0$, and either $n_0=1$, and
$(p,\overline{y})$, $(q,\overline{y}) \in 
 \varepsilon\mathcal{I}_{\alpha, \beta}$, or
$n_0=-1$, and $(p,\overline{y})$, $(q,\overline{y}) \in 
  \varepsilon\mathcal{I}_{\beta,\alpha}$.
\end{itemize}
By Propositions \ref{r:hcurvz} and \ref{p:oninterf}, every 
$\mathcal{C}$--edge is calibrable.
\end{Definition}

\begin{rk}[Symmetric $\mathcal{C}$--edges] \label{e:ln}
When $L$ has positive $\varphi$--curvature, and
$L=[-x_N,x_N]\times \{\overline{y}\}$, with $\overline{y}\in\R$ and $x_N$ 
defined in 
\eqref{f:xenne}, the velocity of $L$ is given by
\begin{equation}\label{f:vln}
v_L=
\frac{1}{x_N}+\frac{\alpha+\beta}{2}-
\frac{(\beta-\alpha)\varepsilon}{8 x_N}.
\end{equation}

On the other hand,  if $L$ has negative $\varphi$--curvature, and
$L=[-x_N-\varepsilon/2,x_N+\varepsilon/2]\times \{\overline{y}\}$, setting 
$\overline{x}= 
x_N+\varepsilon/2$, the velocity of $L$ is given by
\begin{equation}
v_L=
-\frac{1}{\overline{x}}+\frac{\alpha+\beta}{2}+
\frac{(\beta-\alpha)\varepsilon}{8\overline{x}}.
\end{equation}
\end{rk}

The general result concerning the long edges 
with positive $\varphi$--curvature is the following.

\begin{prop}\label{p:gencalibr}
Let $L=[p,q]\times \{\overline{y}\}$ be a horizontal edge with positive 
$\varphi$--curvature, and such that $\ell+\ell_\alpha-\ell_\beta > 
4/(\beta-\alpha)$. Then the following hold.
\begin{itemize}
\item[(i)] If either $g_\varepsilon (p)=\beta$, or $g_\varepsilon 
(q)=\beta$, or $p\in \varepsilon\mathcal{I}_{\alpha,\beta}$, or
$q\in \varepsilon\mathcal{I}_{\beta,\alpha}$, 
then $L$ is not calibrable.
\item[(ii)] If $g_\varepsilon (p)=g_\varepsilon (q) =\alpha$, let 
$\sigma_1$, $\sigma_2\in (0, \varepsilon/2)$ be such that 
$p+\varepsilon/2+\sigma_1 \in \varepsilon \mathcal{I}_{\beta,\alpha}$ and 
$q-\varepsilon/2-\sigma_2 \in \varepsilon \mathcal{I}_{\alpha,\beta}$, and 
let $\tilde{\ell}$ be the length of the interval $[p+\varepsilon/2+\sigma_1,
q-\varepsilon/2-\sigma_2]$.
Setting 
\[
m= \varepsilon 
\frac{\beta-\alpha}{(\beta-\alpha)(\tilde{\ell}+\varepsilon/2)+4},
\qquad 
h=\frac{\varepsilon}{2}\frac{(\beta-\alpha)(\tilde{\ell}+\varepsilon/2)-4}
{(\beta-\alpha)(\tilde{\ell}+\varepsilon/2)+4},
\] and  
\[
\Sigma=\left\{m\sigma_2 +h \leq  
\sigma_1 \leq 
\frac{1}{m} \sigma_2- \frac{h}{m}\right\},
\]
we have $m\in (0,1)$, $\Sigma \cap [0,\varepsilon/2]^2 \neq \emptyset$, and 
$L$ is calibrable with velocity
\[
v_L=\frac{2}{\ell}+\frac{\alpha+\beta}{2}+\frac{\beta-\alpha}{2\ell}
\left(\frac{\varepsilon}{2}-\sigma_1-\sigma_2\right)
\]
if and only if $(\sigma_1,\sigma_2)\in \Sigma$.
\item[(iii)] if $g_\varepsilon(p)=\alpha$, and $q\in \varepsilon 
\mathcal{I}_{\alpha,\beta}$ (resp. $p\in \varepsilon \mathcal{I}_{\beta, 
\alpha}$, and $g_\varepsilon(q)=\alpha$), let $\sigma\in (0,\varepsilon/2)$
be such that $p+\sigma+\varepsilon/2\in \varepsilon \mathcal{I}_{\beta,\alpha}$,
let $\ell^*$ be the length of the interval $[p+\varepsilon/2+\sigma, q]$ (
resp. of $[p, q-\varepsilon/2-\sigma]$), and let
\[
\sigma^*= \frac{\varepsilon}{2}\frac{(\beta-\alpha)(\ell^*+\varepsilon/2)-4}
{(\beta-\alpha)(\ell^*-\varepsilon/2)+4}.
\]
Then $L$ is calibrable if and only if $\sigma \geq \sigma^*$.
\end{itemize}
\end{prop}
\begin{proof}
If $\ell+\ell_\alpha-\ell_\beta > 4/(\beta-\alpha)$, by \eqref{f:conden3}
the candidate Cahn--Hoffmann field $n$ is strictly decreasing in the $\beta$ 
phase.
Hence, under the assumptions in (i), 
$n$ does not satisfy the constraint $|n|\leq 1$ at least near an 
endpoint, and $L$ is not 
calibrable.

Assume now that both the endpoints belong to the $\alpha$ phase,  and let 
$\sigma_1$, $\sigma_2\in (0, \varepsilon/2)$ be such that 
$p+\varepsilon/2+\sigma_1 \in \varepsilon \mathcal{I}_{\beta,\alpha}$ and 
$q-\varepsilon/2-\sigma_2 \in \varepsilon \mathcal{I}_{\alpha,\beta}$.
Then, by \eqref{f:conden3}, we have 
\[
n(p+\varepsilon/2+\sigma_1)-n(p) =   
\sigma_1 \left(\frac{2}{\ell}+\frac{\beta-\alpha}{2 
\ell}(\ell+\ell_\beta-\ell_\alpha)\right)+ \frac{\varepsilon}{2}
\left(\frac{2}{\ell}-\frac{\beta-\alpha}{2 
\ell}(\ell+\ell_\alpha-\ell_\beta)\right).
\] 
In this case we have $ \ell_\beta-\ell_\alpha=\varepsilon/2-\sigma_1-\sigma_2$,
and $\ell=\tilde{\ell}+\varepsilon+\sigma_1+\sigma_2$, so that
\[
\begin{split}
n(p+\varepsilon/2+\sigma_1) & -n(p) \\
& =
\sigma_1 \left(\frac{2}{\ell}+\frac{\beta-\alpha}{2 
\ell}\left(\tilde{\ell}+\frac{3}{2} 
\varepsilon\right)\right)+ \frac{\varepsilon}{2}
\left(\frac{2}{\ell}-\frac{\beta-\alpha}{2 
\ell}\left(\tilde{\ell}+\frac{\varepsilon}{2}+2\sigma_1+2\sigma_2\right)\right) 
\\
& =
\frac{1}{2 
\ell}\left[\sigma_1\left(4+(\beta-\alpha)\left(\tilde{\ell}+
\frac{\varepsilon}{2}\right)\right)+
\frac{\varepsilon}{2}
\left(4-(\beta-\alpha)\left(\tilde{\ell}+\frac{\varepsilon}{2}+2\sigma_2\right)
\right)
\right].
\end{split}
\]
Hence, if $\sigma_2 \geq \sigma_1 \geq m\sigma_2+h$, we have 
\[
n(q-\varepsilon/2-\sigma_2)-n(q)\geq n(p+\varepsilon/2+\sigma_1)-n(p)\geq 0,
\]
and $n$ satisfies the constraint both in $[p,p+\varepsilon/2+\sigma_1]$ and 
in $[ q-\varepsilon/2-\sigma_2,q]$. By Remark \ref{r:incrper}, we obtain that
$|n|\leq 1$ on $L$, and hence
$L$ is calibrable. By symmetry, we obtain the same result when
$\sigma_1 \geq \sigma_2 \geq m\sigma_1+h$, and the conclusion follows.

Finally, we have that $m\sigma_2 +h =  \sigma_1 =\frac{1}{m} \sigma_2- 
\frac{h}{m}$ for $\sigma_1=\sigma_2=\tilde{\sigma}$, where
\begin{equation}\label{f:tildesig}
\tilde{\sigma}:=
\frac{\varepsilon}{2}\frac{(\beta-\alpha)(\tilde{\ell}+\varepsilon/2)-4}
{(\beta-\alpha)(\tilde{\ell}-\varepsilon/2)+4},
\end{equation}
and $\tilde{\sigma}\in (0,\varepsilon/2)$ under the assumption 
$\ell+\ell_\alpha-\ell_\beta > 4/(\beta-\alpha)$, so that 
$\Sigma \cap [0,\varepsilon/2]^2 \neq \emptyset$.

The proof of (iii) follows the same arguments. 
\end{proof}

\begin{figure}[h!]
\includegraphics[height=4cm]{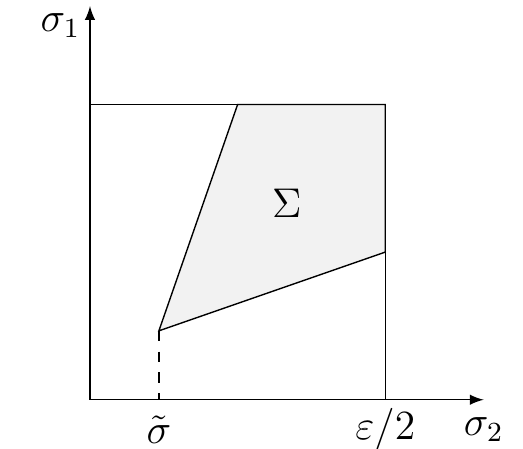}
\caption{The calibrability set $\Sigma$ in Proposition 
\ref{p:gencalibr}}\label{fig:sigma}
\end{figure}

\begin{rk}[Calibrability threshold]\label{r:symm}
In the special case when $\sigma_1=\sigma_2=\sigma>0$, Proposition 
\ref{p:gencalibr}(ii)
implies that $L$ is calibrable if and only if
$\sigma\geq \tilde{\sigma}$, where $\tilde{\sigma}$ is the quantity defined in 
\eqref{f:tildesig}.
When $\sigma=\tilde{\sigma}$, the edge $L$ is calibrated by a Cahn--Hoffmann
vector field $n$ such that $n(p)=n(p+\varepsilon/2+\tilde{\sigma})$ and 
$n(q)=n(q-\varepsilon/2-\tilde{\sigma})$. As a consequence,
the same field calibrates both the edges 
$[p,p+\varepsilon/2+\tilde{\sigma}]\times 
\{\overline{y}\}$
and $[q-\varepsilon/2-\tilde{\sigma},q]\times \{\overline{y}\}$
(as edges with zero $\varphi$--curvature, see Proposition \ref{r:hcurvz}),
and the edge 
$[p+\varepsilon/2+\tilde{\sigma},q-\varepsilon/2-\tilde{\sigma}]\times 
\{\overline{y}\}$ (as edges with positive $\varphi$--curvature, see Proposition 
\ref{p:oninterf}) with the same velocity.

Similarly, in the case (iii) Proposition \ref{p:gencalibr}, when 
$\sigma=\sigma^*$, the edge $L$ is calibrated by a Cahn--Hoffmann
vector field $n$ such that $n(p)=n(p+\varepsilon/2+\sigma^*)$, and 
the same field calibrates both the edge 
$[p,p+\varepsilon/2+\sigma^*]\times 
\{\overline{y}\}$
(as edge with zero $\varphi$--curvature),
and the edge 
$[p+\varepsilon/2+\tilde{\sigma},q]\times 
\{\overline{y}\}$ (as edges with positive $\varphi$--curvature) with the same 
velocity.
\end{rk}

 \begin{rk}[Symmetric edges with positive $\varphi$--curvature]\label{r:delta}
 If $L=[-\ell/2, \ell/2]\times \{\overline{y}\}$, $\ell>0$, 
  $\overline{y}\in\R$, is a 
  horizontal edge with positive 
  $\varphi$--curvature, the previous results reed as follows. 
 Setting $x_N=(N+1/4)\varepsilon$, $N\in\N$, let us define
  $\overline{N}_\varepsilon\in\N$ by
 \begin{equation}\label{f:ovenne}
 \left(2x_N+\frac{\varepsilon}{2}\right)> \frac{4}{\beta-\alpha}, \qquad 
 \forall\ N \geq \overline{N}_\varepsilon,
 \end{equation}
and 
 $\delta(N)\in(0,\varepsilon)$ by
 \begin{equation}\label{f:delta}
 \delta(N)=\dfrac{x_N(\beta-\alpha) \varepsilon}{2+\left(2x_N- 
 \dfrac{\varepsilon}{2}\right)\dfrac{(\beta-\alpha)}{2}},
 \end{equation}
 then
 $\delta(N)>\varepsilon/2$ if and only if 
 $N\geq \overline{N}_\varepsilon$. 
 Hence, if $\ell/2=x_{N_\varepsilon(\ell)}+\delta(\ell)$, with
$0\leq \delta(\ell)  <  \varepsilon$, and  
 \begin{equation}\label{f:defnelle}
 N_\varepsilon(\ell)=\left\lfloor\dfrac{\ell}{2\varepsilon}-
 \dfrac{1}{4}\right\rfloor,  
 \end{equation}
 the following hold.
 \begin{itemize}
 \item[(i)] If $N_\varepsilon(\ell) < \overline{N}_\varepsilon$, then the edge 
 $L$ is 
 calibrable.
 \item[(ii)] If $N_\varepsilon(\ell) \geq \overline{N}_\varepsilon$ then the 
 edge $L$ is 
 calibrable if 
 and 
 only if  $\delta(\ell) \geq \delta(N_\varepsilon(\ell))$.
 \end{itemize} 

\end{rk}
  
The analogous of Proposition \ref{p:gencalibr} concerning the long 
edges 
with negative $\varphi$--curvature is the following.

\begin{prop}\label{p:gencalibrneg}
Let $L=[p,q]\times \{\overline{y}\}$ be a horizontal edge with negative 
$\varphi$--curvature, and such that $\ell+\ell_\beta-\ell_\alpha > 
4/(\beta-\alpha)$. Then the following hold.
\begin{itemize}
\item[(i)] If either $g_\varepsilon (p)=\alpha$, or $g_\varepsilon 
(q)=\alpha$, or $p\in \varepsilon\mathcal{I}_{\beta,\alpha}$, or
$q\in \varepsilon\mathcal{I}_{\alpha,\beta}$, 
then $L$ is not calibrable.
\item[(ii)] If $g_\varepsilon (p)=g_\varepsilon (q) =\beta$, let 
$\sigma_1$, $\sigma_2\in (0, \varepsilon/2)$ be such that 
$p+\varepsilon/2+\sigma_1 \in \varepsilon \mathcal{I}_{\alpha,\beta}$ and 
$q-\varepsilon/2-\sigma_2 \in \varepsilon \mathcal{I}_{\beta,\alpha}$, and 
let $\tilde{\ell}$ be the length of the interval $[p+\varepsilon/2+\sigma_1,
q-\varepsilon/2-\sigma_2]$. 

Then there exist $m\in (0,1)$ and $h>0$ such that
$L$ is calibrable with velocity
\[
v_L=-\frac{2}{\ell}+\frac{\alpha+\beta}{2}-\frac{\beta-\alpha}{2\ell}
\left(\frac{\varepsilon}{2}-\sigma_1-\sigma_2\right)
\]
if and only if $(\sigma_1,\sigma_2)\in \left\{m\sigma_2 +h \leq  
\sigma_1 \leq 
\frac{1}{m} \sigma_2- \frac{h}{m}\right\}$.

\item[(iii)] if $g_\varepsilon(p)=\beta$, and $q\in \varepsilon 
\mathcal{I}_{\beta,\alpha}$ (resp. $p\in \varepsilon 
\mathcal{I}_{\alpha, 
\beta}$, and $g_\varepsilon(q)=\beta$), let $\sigma\in (0,\varepsilon/2)$
be such that $p+\sigma+\varepsilon/2\in \varepsilon \mathcal{I}_{\alpha,\beta}$,
let $\ell^*$ be the length of the interval $[p+\varepsilon/2+\sigma, q]$ (
resp. of $[p, q-\varepsilon/2-\sigma]$), and let
\[
\sigma^*= \frac{\varepsilon}{2}\frac{(\beta-\alpha)(\ell^*+\varepsilon/2)-4}
{(\beta-\alpha)(\ell^*-\varepsilon/2)+4}.
\]
Then $L$ is calibrable if and only if $\sigma \geq \sigma^*$.
\end{itemize}
\end{prop}

\section{Effective motion as $\varepsilon\to 0$}\label{s:eff}

\subsection{Evolution of rectangles}\label{s:rect}

This section is devoted to the proof of the following result.

\begin{teor}[Effective motion of coordinate rectangles]\label{t:rect}
Let $R_0$ be a coordinate rectangle, and let
$\ell_{1,0}$, $\ell_{2,0}$ be the length of its horizontal and vertical edges, 
respectively. For every $\varepsilon>0$, let $R_0^\varepsilon$ be a coordinate 
rectangle such that
$
d_H(R_0,R_0^\varepsilon)<\varepsilon.
$
Then there exists a variational crystalline curvature flow of 
$R_0^\varepsilon$ with forcing term $g_\varepsilon$. Moreover,  every 
variational crystalline curvature flow $E^\varepsilon(t)$  of 
$R_0^\varepsilon$
converges, in the Hausdorff topology and 
locally uniformly in time, as $\varepsilon \to 0$ to the coordinate rectangle 
$R(t)$ whose horizontal and vertical edges have lengths
$\ell_{1}(t), \ell_{2}(t)$ solving the system of ODEs   
\begin{equation}\label{f:evofin}
\begin{cases}
\ell_1'=-2\, H_g(\ell_2), \\
\ell_2'=-\dfrac{4}{\ell_1}-(\alpha+\beta),
\end{cases}
\end{equation}
with initial datum $(\ell_{1}(0), \ell_{2}(0))=(\ell_{1,0}, 
\ell_{2,0})$. The function $H_g(\ell)\colon (0,+\infty) \to \R$ is a 
truncation of the harmonic mean defined by
\begin{equation}\label{f:harmean}
H_g(\ell):=
\begin{cases}
0, & \text{if}\ \ell\geq -\dfrac{2}{\alpha}, \\
\left\langle \dfrac{2}{\ell} +g\right\rangle = 
\dfrac{1}{ \displaystyle 
\int_0^1 \frac{1} {2/\ell_2+g(s)}\, ds}= \dfrac{(2 
+\alpha\ell_2)(2+\beta 
\ell_2)}{\ell_2\left(2+\dfrac{\alpha+\beta}{2} \ell_2\right)}, & 
\text{otherwise}.
\end{cases}
\end{equation}
\end{teor} 

\begin{rk} \label{r:odedis}
If we assume that the evolution $E^\varepsilon(t)$ of a coordinate rectangle 
$R_0^\varepsilon$ is a coordinate rectangle for $t\in[0,T]$, and we denote 
by $(x_1(t),y_1(t))$, $(x_2(t),y_1(t))$, $(x_2(t),y_2(t))$, $(x_1(t),y_2(t))$,
with $x_1(t) < x_2(t)$, and $y_1(t) > y_2(t)$, the coordinates of the 
vertices 
of 
$E^\varepsilon(t)$, the evolution of these points is governed by the system 
of  ODEs'
\begin{equation}\label{f:rrhsdis}
\begin{cases}
x_1'= \dfrac{2}{y_1-y_2}+g_\varepsilon(x_1) \\[10pt]
x_2'= -\dfrac{2}{y_1-y_2}-g_\varepsilon(x_2) \\[10pt]
y_1'= -\dfrac{2}{x_2-x_1}- \dfrac{\alpha+\beta}{2}- 
h_\varepsilon(y_1,y_2) \\[10pt]
y_2'= \dfrac{2}{x_2-x_1}+ \dfrac{\alpha+\beta}{2}+
h_\varepsilon(y_1,y_2)
\end{cases}
\end{equation}
in the domain $D:=\{(x_1,x_2,y_2,y_2) \colon x_1< x_2,\ y_1>y_2\}$. 
The Lipshitz function 
$h_\varepsilon\colon \{ y_1<y_2\}\subseteq \R^2 \to \R$ takes into 
account the 
small reminder varying
in $[-\varepsilon/2,\varepsilon/2]$ and appearing in \eqref{f:velocita}.

Applying the classical results of differential equations with discontinouos 
right--hand side (see \cite{Fi}, Chapter 2), we obtain the following 
properties for the solutions.

\begin{itemize}
\item[(i)] For every $P_0\in D$ there exists a (Filippov) solution to 
\eqref{f:rrhsdis}
starting from $P_0$.

\item[(ii)] For every $P_0\in D\cap \{y_1-y_2 \leq -2/\alpha\}$ there 
exists
a unique local solution to \eqref{f:rrhsdis}
starting from $P_0$, and defined as long as it satisfies
$y_1(t)-y_2(t)< -2/\alpha$.

\item[(iii)] If $P_0\in D\cap\{y_1-y_2 > -2/\alpha\}$, the uniqueness 
of 
the solution starting from $P_0$ fails if and only if either 
$x_1$ or $x_2$ belongs to the set of ``unstable discontinuities'' 
$U=\{\pm(x_N+\varepsilon/2), N\in\N\}$ (see \eqref{f:defnelle} for the 
definition of $x_N$). 
If this does not occur, then 
the solution is unique until the first time $t_0$ for which 
either $x_1(t_0)$ or $x_2(t_0)$ belong to $U$.

\item[(iv)] If $P_0\in D\cap\{y_1-y_2 > -2/\alpha\}$, and  
$x_1$ (resp.  $x_2$) belongs  to the set of ``stable discontinuities''
$S=\{\pm x_N, N\in\N\}$, then $x_1'(t)=0$ (resp.  $x_2'(t)=0$) as long 
as 
the solution satisfies $y_1(t)-y_2(t) > -2/\alpha$.
\end{itemize}

Based on results of Section \ref{s:calibrability}, a 
coordinate rectangle may not be calibrable, so that we cannot expect the 
evolution 
to preserve the geometry of the initial datum. In the proof of Theorem 
\ref{t:rect} we combine the previous properties of the solutions to 
\eqref{f:rrhsdis} with a carefull description
of why and how the geometry changes during the evolution.
\end{rk}

\begin{proof}[Proof of Theorem \ref{t:rect}]
We first assume that $R_0$ is centered at the origin. Using the 
notation
\[
R(\ell_1,\ell_2)=\left[-\lum, \lum\right]\times \left[-\ldm,\ldm\right].
\] 
for this type of rectangles, $R_0=R(\ell_{1,0},\ell_{2,0})$.

We also assume that $R_0^\varepsilon=R(2x_{N_\varepsilon},\ell_{2,0})$, where 
${N_\varepsilon}={N_\varepsilon}(\ell_{1,0})$ is defined in \eqref{f:defnelle}.
In this case, recalling that the calibrability of a coordinate rectangle 
depends only 
on the calibrability of its horizontal edges (see Section 
\ref{rk:vert}), 
we have that $R_0^\varepsilon$ is calibrable (see Remark \ref{e:ln}),
and the evolution starts according to \eqref{f:rrhsdis}.

\medskip

\noindent{Case 1 (homogenized velocities): {$\ell_{1,0}< 
4/(\beta-\alpha)$}, 
$\ell_{2,0} \leq -2/\alpha$.}
The vertical edges of 
$R_0^\varepsilon=R(2x_{N_\varepsilon},\ell_{2,0}^\varepsilon)$
are calibrable with velocity $v_2^\varepsilon(0)\geq  0$, and, by  
Remark 
\ref{e:ln}, the 
horizontal edges are also calibrable with velocity
\[
v_1^\varepsilon(0)=\frac{1}{x_{N_\varepsilon}}+\frac{\alpha+\beta}{2}+
\frac{(\alpha-\beta)\varepsilon}{8x_{N_\varepsilon}}> 0.
\]
Then, by Section 
\ref{rk:vert}, Proposition 
\ref{p:fracon}, and Remark \ref{r:odedis}(ii), 
the (unique) evolution is given by shrinking coordinate rectangles  
$R(\ell_1^\varepsilon(t)),\ell_2^\varepsilon(t))$, and it 
is governed by system \eqref{f:rrhsdis}, or equivalently, by
\begin{equation}\label{f:appc1}
\begin{cases}
{\ell_1^{\varepsilon}}'= -\dfrac{4}{\ell_2^\varepsilon}-2g
\left(\dfrac{\ell_1^\varepsilon}{\varepsilon}\right), \\[8pt]
{\ell_2^\varepsilon}'=-\dfrac{4}{\ell_1^\varepsilon}-(\alpha+\beta)-
\dfrac{(\alpha-\beta)\varepsilon}{2\ell_1^\varepsilon}.
\end{cases}
\end{equation} 

In order to pass to the limit as $\varepsilon \to 0^+$ and to find the 
effective evolution, notice that
$\ell_i^\varepsilon\in (\ell_{i,\beta}, \ell_{i,\alpha})$, 
$i=1,2$, 
where $(\ell_{1,\alpha}, \ell_{2,\alpha})$ $(\ell_{1,\beta}, 
\ell_{2,\beta})$ are the solution to \eqref{f:forzc} with initial 
datum $(2x_{N_\varepsilon},\ell_{2,0}^\varepsilon)$, and forcing term 
$\gamma=\alpha$ and $\gamma=\beta$ respectively. 
Hence, for $T>0$, $\ell_i^\varepsilon$ are equilipschitz in 
$[0,T]$. Let $(\ell_1(t), \ell_2(t))$ the uniform limit of a suitable 
subsequence of $(\ell_1^\varepsilon(t)),\ell_2^\varepsilon(t))$ in [0,T].
Being a uniform limits of monotone non--increasing functions, 
$\ell_i(t)$ are 
non--increasing and hence differentiable almost everywhere in $[0,T]$. 
Moreover, for $t\in (0,T)$ and $\sigma>0$ such that $t+\sigma < T$, we have 
$\ell_2^\varepsilon (\tau)= \ell_2^\varepsilon (t)+O(\sigma)$ for $\tau\in 
[t,t+\sigma]$, and hence
\[
\int_{\ell_1^\varepsilon(t)}^{\ell_1^\varepsilon(t+\sigma)} 
\dfrac{1}{2/\ell_2^\varepsilon(t)+ 
g\left(s/\varepsilon\right)}\, ds +o(\sigma) =-2 \sigma
\]
and a passage to the limit as $\varepsilon \to 0$ gives
\begin{equation}\label{f:armean}
\int_{\ell_1(t)}^{\ell_1(t+\sigma)} \frac{1}{\left\langle 
2/\ell_2(t)+g 
\right\rangle}\, ds + o(\sigma) =-2 \sigma
\end{equation}
where $\langle 2/\ell_2+g \rangle$ is the harmonic mean of 
$2/\ell_2+g$ in 
$[0,1]$.

If $t$ is a differentiability point for $\ell_1$, 
from \eqref{f:armean} it follows that
$\ell_1'(t)=-2 \langle 2/\ell_2(t)+g \rangle$.
In conclusion, the effective evolution of the rectangle 
$R(\ell_{1,0},\ell_{2,0})$
is given by rectangles $R(\ell_1(t),\ell_2(t))$ satisfying  the evolution law
\begin{equation}\label{f:caseone}
\begin{cases}
\ell_1'=-2\left\langle 
\dfrac{2}{\ell_2}+g 
\right\rangle , \\[8pt]
\ell_2'=-\dfrac{4}{\ell_1}-(\alpha+\beta).
\end{cases}
\end{equation}
In this case all the edges move inwards until a finite extinction time.

\medskip

\noindent{Case 2 (mesoscopic pinning): 
{$\ell_{2,0}\geq-2/\alpha$}.} 
By 
Section 
\ref{rk:vert}, 
the vertical edges of $R_0^\varepsilon=R(2x_{N_\varepsilon},\ell_{2,0})$
are calibrable with velocity $v_2^\varepsilon(0)=0$. Moreover, by Proposition 
\ref{p:oninterf}, the horizontal edges of 
$R_0^\varepsilon=(x_{N_\varepsilon},\ell_{2,0})$ are also calibrable with 
velocity
\[
v_1^\varepsilon(0)=\frac{1}{x_{N_\varepsilon}}+\frac{\alpha+\beta}{2}+
\frac{(\alpha-\beta)\varepsilon}{8x_{N_\varepsilon}}.
\]
If $v_1^\varepsilon(0)\leq 0$, then the length of the vertical edges 
does not decrease and 
the (unique) evolution at the mesoscopic scale is given by rectangles 
$R^\varepsilon(\ell_1^\varepsilon(t),\ell_2^\varepsilon(t))$, $t>0$, where
\[
\begin{cases}
{\ell_1^\varepsilon}'=0, \\
{\ell_2^\varepsilon}'= -\dfrac{2}{x_{N_\varepsilon}}-(\alpha+\beta)-
\dfrac{(\alpha-\beta)\varepsilon}{4x_{N_\varepsilon}}.
\end{cases}
\]
Taking the limit as $\varepsilon \to 0$ we obtain that the effective 
evolution
is given by $R(\ell_1(t),\ell_2(t))$, $t>0$, where
\begin{equation}\label{f:pinn}
\begin{cases}
\ell_1'=0, \\
\ell_2'= -\dfrac{4}{\ell_{1,0}}-(\alpha+\beta).
\end{cases}
\end{equation}
Hence, if $\alpha+\beta<0$,  the rectangles 
$R(-4/(\alpha+\beta),\ell_2)$ with 
$\ell_{2,0}\geq-2/\alpha$ are unstable equilibria, while,
if  $v_0=4/\ell_{1,0}+(\alpha+\beta)<0$, the rectangle expands 
in 
the vertical direction with constant velocity $v_0$, keeping the 
length of the horizontal edges fixed.

If, instead, $v_0>0$, then the horizontal edges start to move 
inward, so that the length of the vertical edges decreases, and 
\eqref{f:pinn} 
describes the evolution for $t\in [0,\overline{t}]$, where
\[
\overline{t}=\sup\{t>0\colon\ \ell_2(s))\geq-2/\alpha\ \forall s\in[0,t)\}.
\]
Starting from $R(\ell_1(\overline{t}),-2/\alpha)$, the 
evolution is the one shown in Case 1 or 3, respectively.

\medskip 

\noindent{Case 3 (mesoscopic breaking): 
{$\ell_{1,0}\geq 4/(\beta-\alpha), 
\ \ell_{2,0} \leq -2/\alpha$}.}
By Section 
\ref{rk:vert}, 
the vertical edges of 
$R_0^\varepsilon=R(2x_{N_\varepsilon},\ell_{2,0})$
are calibrable with velocity $v_2(0)\geq 0$, and,  by Remark 
\ref{e:ln}, the 
horizontal edges are calibrable with velocity
\[
v_1^\varepsilon(0)=\frac{1}{x_{N_\varepsilon}}+\frac{\alpha+\beta}{2}+
\frac{(\alpha-\beta)\varepsilon}{8x_{N_\varepsilon}}.
\]
By Remark 
\ref{r:delta}, the evolution is a rectangle, with decreasing length of the 
horizontal edges $\ell_{1}^\varepsilon(t)$, until the time $t_\delta>0$ 
such that
$\ell_{1}^\varepsilon(t_\delta)=2x_{N_\varepsilon}+2\delta(N_\varepsilon-1)$, 
where $\delta(N_\varepsilon-1)$ is the calibrability threshold given 
in \eqref{f:delta}. 
The horizontal edges of the rectangle  
cannot be calibrable after the time $t_\delta$. Nevertheless,
by Remark \ref{r:symm}, the Cahn--Hoffman vector field calibrating the 
horizontal edges at time $t_\delta$ equals the one calibrating 
separately the horizontal edges with positive $\varphi$--curvature
\[
E_c^{\pm} := [-x_{N_\varepsilon-1}, x_{N_\varepsilon-1}]\times \{\pm 
\ell_2^\varepsilon(t_\delta)/2\}
\]
with velocity 
\[
v_c^\varepsilon=\frac{1}{x_{N_\varepsilon-1}}+\frac{\alpha+\beta}{2}+
\frac{(\alpha-\beta)\varepsilon}{8x_{N_\varepsilon-1}},
\] 
and the horizontal edges with zero curvature 
\[
\begin{split}
E_l^\pm & :=
\left[\frac{-\ell_{1}(t_\delta)}{2},-x_{N_\varepsilon-1}\right]\times 
\left\{\pm 
\frac{\ell_2^\varepsilon(t_\delta)}{2}\right\} \\
E_r^\pm & := 
\left[x_{N_\varepsilon-1},\frac{\ell_{1}(t_\delta)}{2}\right]\times 
\left\{\pm 
\frac{\ell_2^\varepsilon(t_\delta)}{2}\right\}
\end{split}
\]
with the same velocity (see Figure \ref{fig:fig12}). 
\begin{figure}[h!]
\includegraphics[height=5cm]{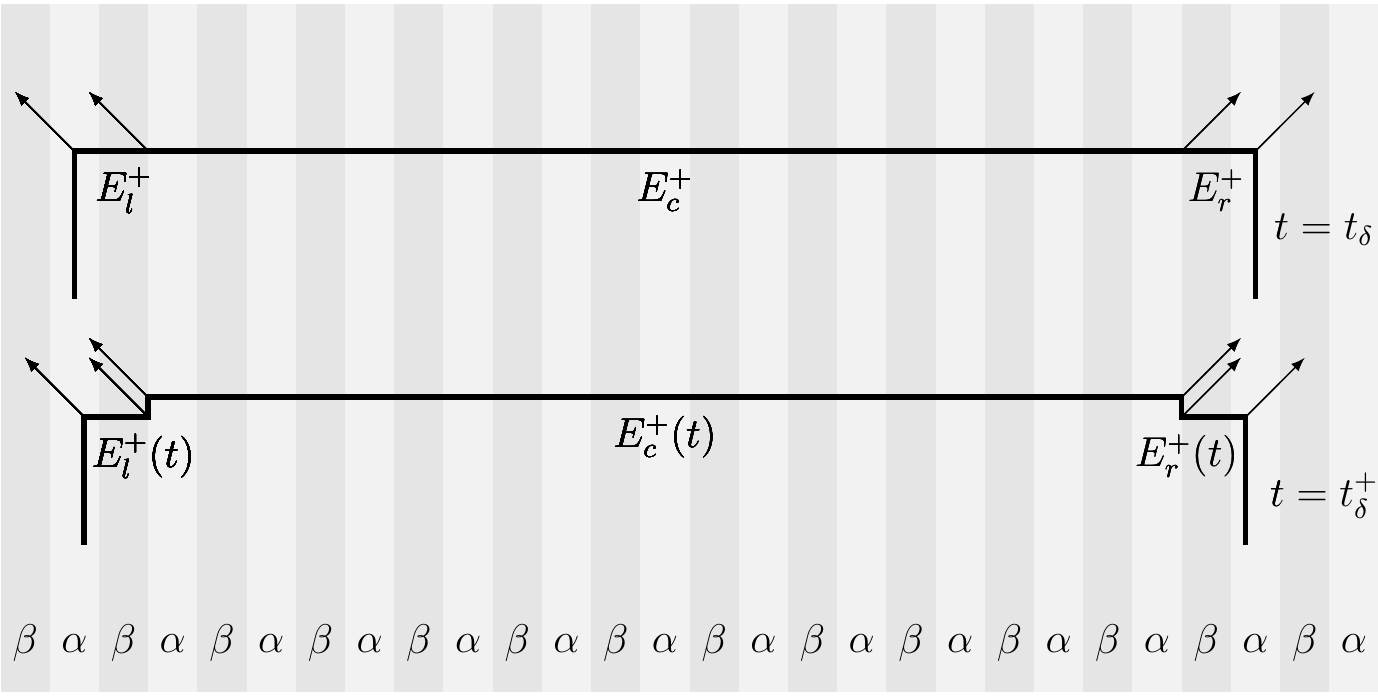}
\caption{Breaking of the horizontal edge at $t=t_\delta$.}\label{fig:fig12}
\end{figure}

Hence, the Cahn--Hoffman 
vector field $n(t)$ evolves continuously after $t_\delta$ by breaking the 
horizontal edges into 
three parts, that we denote by $E_l^\pm(t)$, $E_c^\pm(t)$ and $E_r^\pm(t)$ 
respectively, and the evolution becomes a coordinate polyrectangle 
$P^\varepsilon(t)$  (see Figure \ref{fig:fig12}). 

By symmetry, the lengths and the velocities of the edges $E_l^\pm(t)$, 
$E_c^\pm(t)$ are the same; they will be denoted by 
$\ell_h^\varepsilon(t)$ and 
$v_h^\varepsilon (t)$, respectively.
 
By Section \ref{rk:vert} the small vertical edges with zero 
$\varphi$--curvature are pinned on the 
interfaces $x= \pm x_{N_\varepsilon-1}$, so that, by Remark \ref{e:ln}, the 
long horizontal edges with positive $\varphi$--curvature  $E_c^\pm(t)$
have constant length $\ell_1^\varepsilon = 2x_{N_\varepsilon-1}$, and move 
with velocity
$v_1^\varepsilon(t)= v_c^\varepsilon$.

By Remark \ref{e:latino}, the small horizontal edges $E_l^\pm(t)$ and 
$E_c^\pm(t)$ 
with zero $\varphi$--curvature and length $\ell_h^\varepsilon(t)\in (0, 
\varepsilon/2+ 
\delta(N_\varepsilon))$ move  with velocity 
$v_h^\varepsilon(t)>v_c^\varepsilon(t)$ given by
\[
v_h^\varepsilon(t)=
\begin{cases}
\beta & \text{if}\  \ell_h^\varepsilon(t)\in (0, \varepsilon/2) \\[6pt]
\dfrac{1}{\ell_h^\varepsilon(t)}\left(\dfrac{\varepsilon}{2}\beta+
\left(\ell_h^\varepsilon(t)-\dfrac{\varepsilon}{2}\right)\alpha\right)  &  
\text{if}\   \ell_h^\varepsilon(t)\in(\varepsilon/2, 
\varepsilon)  
\end{cases}
\]
reducing the length $\ell_2^\varepsilon(t)$ of the long vertical edges with 
positive $\varphi$--curvature
\[
\ell_2^\varepsilon(t)=\ell_2^\varepsilon(t_\delta)-\int_{t_\delta}^{t} 
v_h^\varepsilon (s)\, ds \leq \ell_2^\varepsilon(t_\delta)-v_c^\varepsilon 
(t-t_\delta).
\]
On the other hand, the vertical long edges with positive $\varphi$--curvature 
move inward with velocity
\[
v_2^\varepsilon (t)=\frac{2}{\ell_2^\varepsilon(t)}+g_\varepsilon
\geq \frac{2}{\ell_2^\varepsilon(t_\delta)-v_c^\varepsilon 
(t-t_\delta)}+\alpha
\]
so that, if we denote by $t_1$ the time when the vertical long edges with 
positive $\varphi$--curvature reach 
the interfaces $x=\pm x_{N_\varepsilon-1}$, we obtain
\[
2\varepsilon \geq \int_{t_\delta}^{t_1}
\frac{2}{\ell_2^\varepsilon(t_\delta)-v_c^\varepsilon 
(t-t_\delta)}\, dt +\alpha (t_1-t_\delta) \geq \left( 
\frac{2}{\ell_2^\varepsilon(t_\delta)}+\alpha\right)(t_1-t_\delta).
\]
Hence, at $t=t_1$ the evolution is a rectangle $R(x_{N_\varepsilon -1}, 
\ell_2^\varepsilon(t_1))$ where
\[
\ell_2^\varepsilon(t_1)= \ell_2^\varepsilon(t_\delta)-v_c^\varepsilon 
(t_1-t_\delta) = \ell_2^\varepsilon(t_\delta)+ O(\varepsilon), \quad 
\varepsilon \to 0^+.
\]
The (unique) evolution then iterates this ``breaking and recomposing'' motion 
in such
a way that it can be approximate, in Hausdorff topology and locally uniformly 
in 
time, by a family of rectangles $R(\tilde{\ell}_1^\varepsilon (t), 
\tilde{\ell}_2^\varepsilon 
(t))$ satisfying \eqref{f:appc1}, so that the effective motion is a family
of rectangles $R(\ell_1 (t), \ell_2 (t))$  governed by the evolution 
law 
\eqref{f:caseone}.

\medskip

The general results recalled in Remark \ref{r:odedis} can be used to show that 
the effective evolution \eqref{f:evofin} does not depend on the choiche of the 
approximating sequence $R_0^\varepsilon$ of initial data. 

Namely, in Case 1, any coordinate rectangle 
$R_0^\varepsilon$ is calibrable and the (unique) evolution starting from 
$R_0^\varepsilon$ is the family of coordinate rectangles solving 
\eqref{f:rrhsdis}. This evolution has a distance of order $\varepsilon$
from the one starting from
$R(2x_{N_\varepsilon},\ell_{2,0})$ uniformly in time, so that it converges,
as $\varepsilon \to 0^+$ to the same effective evolution.

Similarly, in Case 3, we have that the evolution starting from 
$R_0^\varepsilon$ becomes a 
rectangle with vertical edges with left endpoint on 
$\varepsilon\mathcal{I}_{\beta,\alpha}$ and right endpoint on 
$\varepsilon\mathcal{I}_{\alpha,\beta}$ 
in a time span 
of order $\varepsilon$, possibly breaking and recomposing the horizontal edges 
in the meanwhile. Then, the effective evolution 
of $R_0$ 
is uniquely determined by \eqref{f:caseone}.

On the other hand, if $\ell_{2,0}> -2/\alpha$, then, by
Remark \ref{r:odedis}(iii) and (iv), the position of the 
vertical edges during 
the evolution is confined in the strip $\{x\in 
[x_{N_\varepsilon-1},x_{N_\varepsilon}]\}$, with 
$N_\varepsilon=N_\varepsilon( \ell_{1,0}^\varepsilon)$. 
Hence, at a macroscopic level, the vertical edges are 
pinned, and the effective evolution 
of $R_0$ 
is uniquely determined by \eqref{f:pinn}.

\end{proof}

\begin{figure}[h!]
\includegraphics[width=4cm]{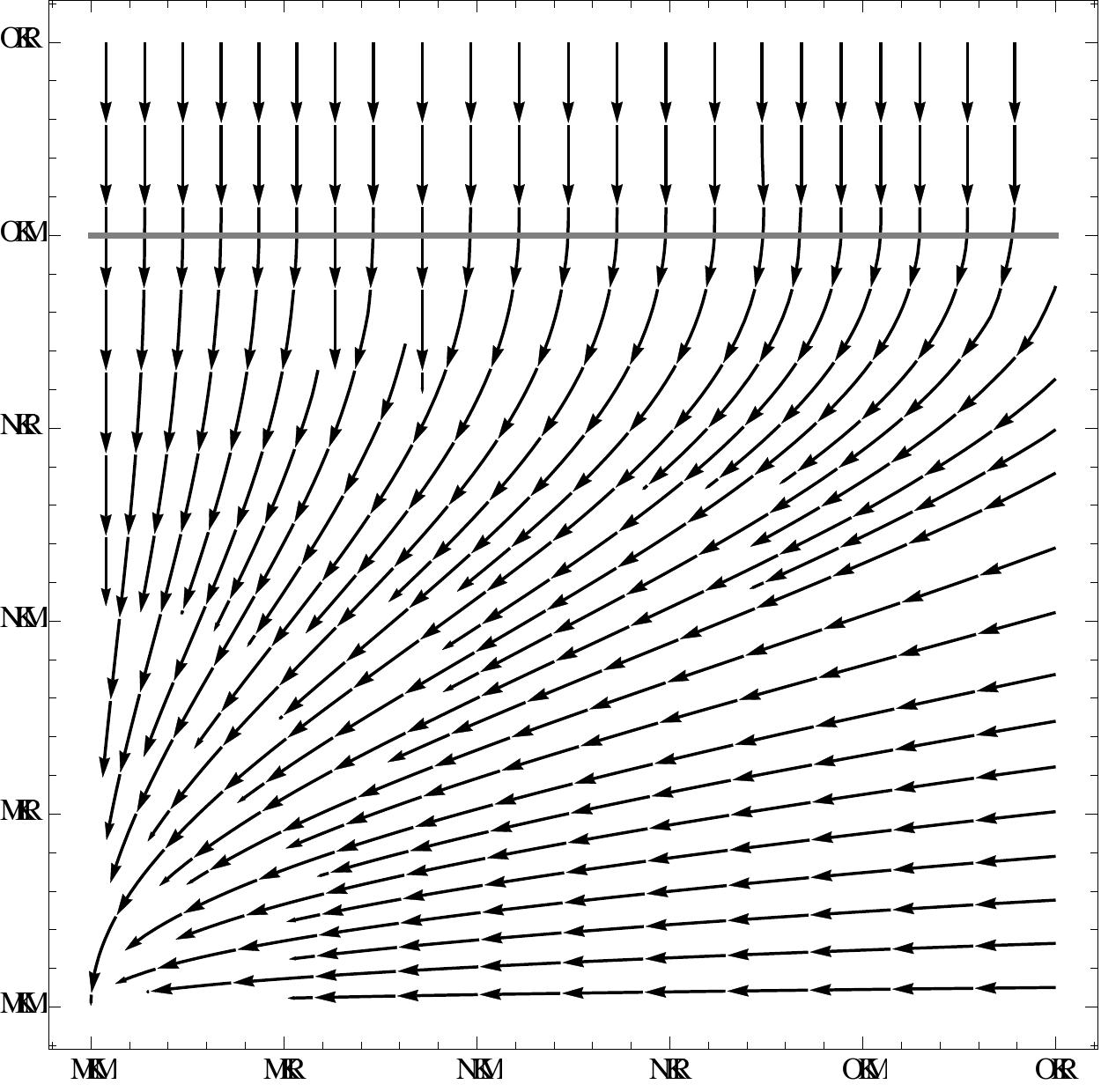}
\quad
\includegraphics[width=4cm]{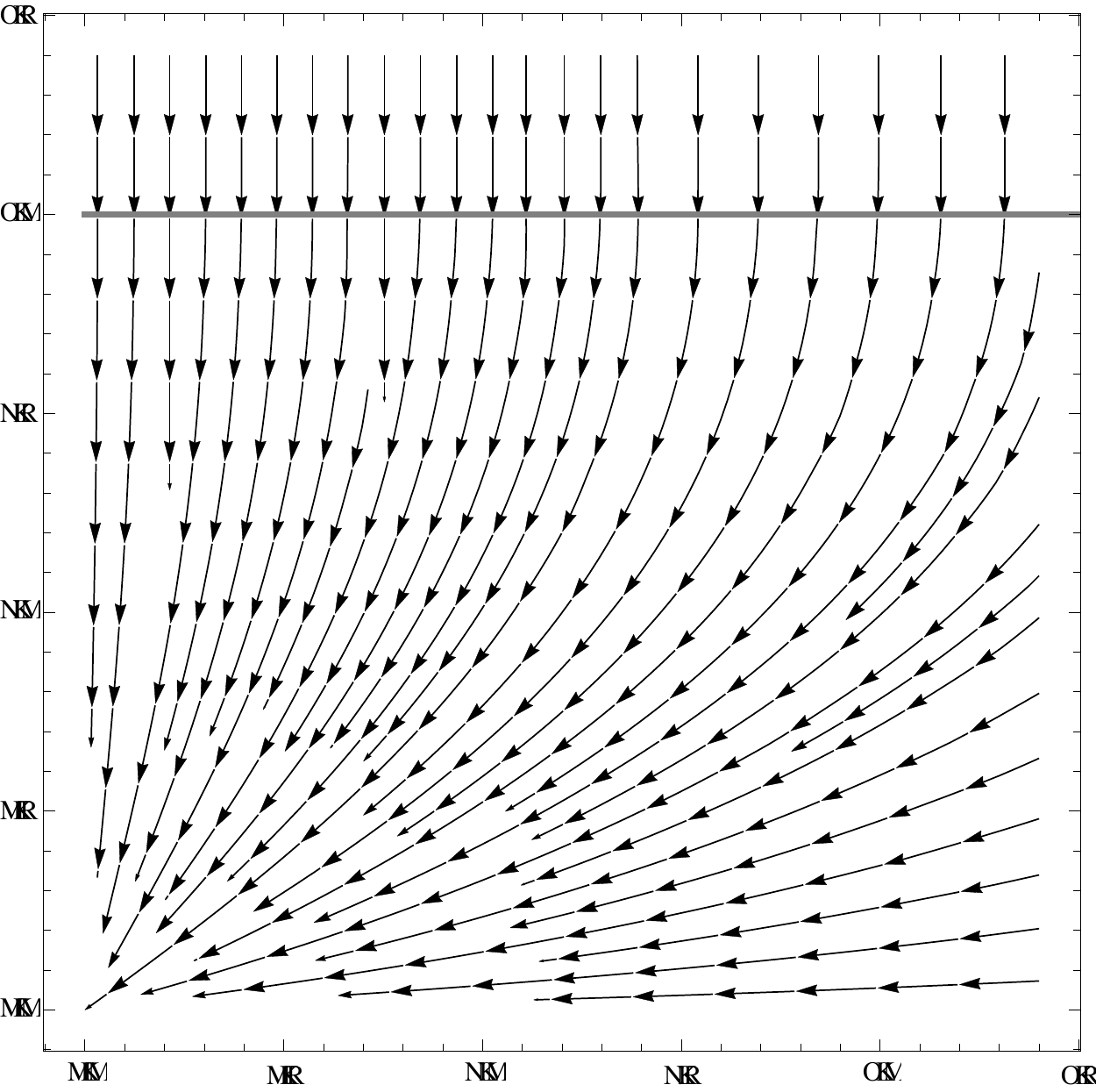}
\quad
\includegraphics[width=4cm]{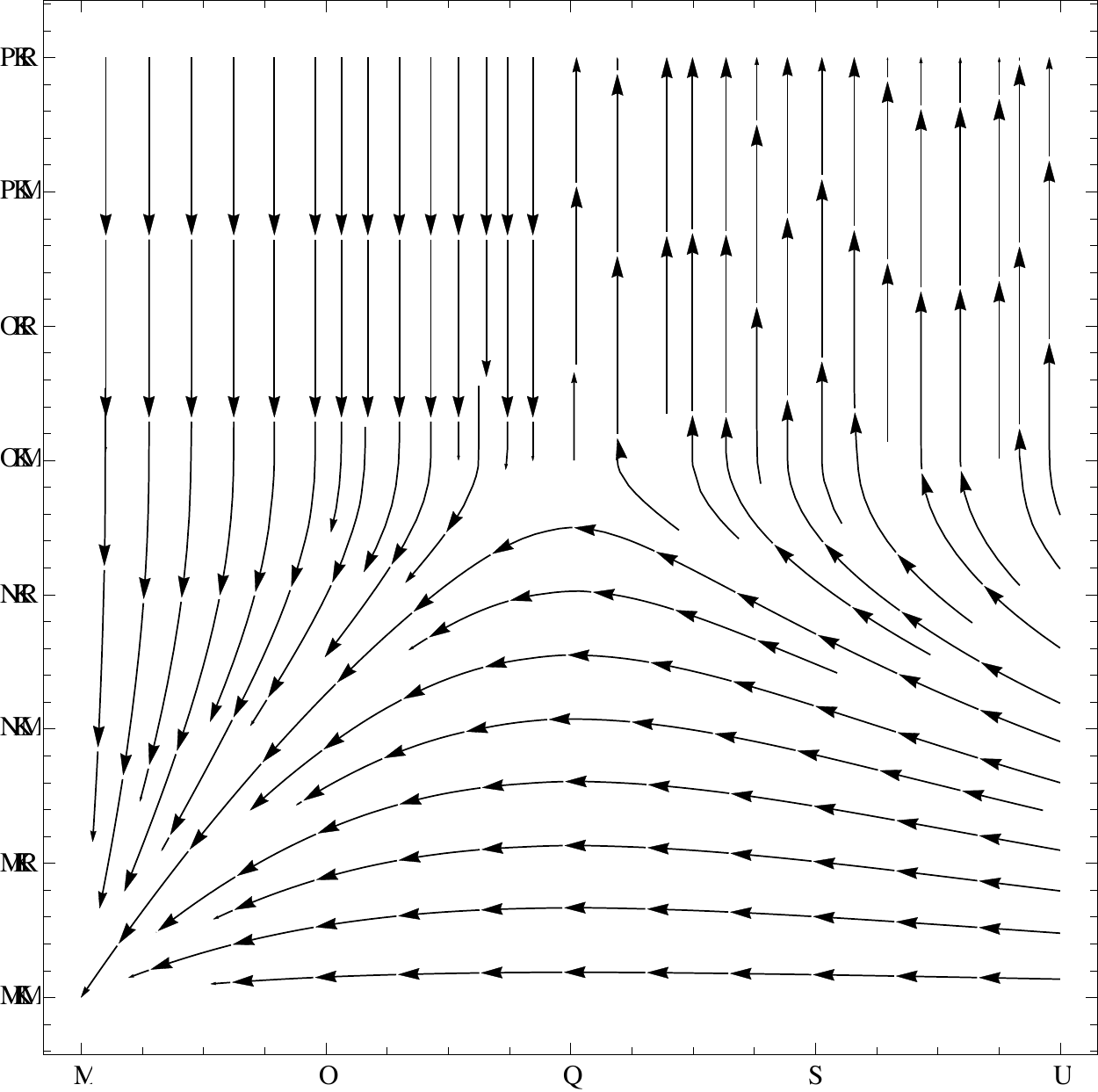}
\caption{Phase portraits of \eqref{f:evofin} for $\alpha+\beta=0$ (left), 
$>0$ (center),
$<0$ (right) }\label{fig:dinamicalim}
\end{figure}

System \eqref{f:evofin} is integrable, and its phase portrait is plotted in 
Figure \ref{fig:dinamicalim}. Notice the pinning effect for long vertical 
edges, and the presence of a half line of nontrivial equilibria for 
$\alpha+\beta<0$.

\begin{rk}[Partial uniqueness of the forced crystalline curvature flow]
We emphasize that, in the proof of Theorem \ref{t:rect}, we have shown that
the forced crystalline curvature flow starting from a coordinate rectangle
of the type $R_0^\varepsilon=R(2x_{N_\varepsilon},\ell_{2,0})$ is unique.
The uniqueness of the solution is essentially based on 
Remark \ref{r:odedis}(iii) and (iv), and it holds for the evolution starting 
from 
any coordinate rectangle with horizontal $\mathcal{C}$--edges (recall 
Definition \ref{d:cedge}).
Namely, by Remark \ref{r:odedis}(iii), uniqueness may fail if and only 
if
the evolution has a ``long'' vertical edge on a unstable discontinuity, but 
this is not the case, because the evolution from a coordinate rectangle with 
horizontal $\mathcal{C}$--edges has pinned ``long'' vertical edges
(see also Proposition \ref{r:uniq} below).
\end{rk}


\subsection{Evolution of polyrectangles}\label{s:poly}

We now extend the previous results to the more general class 
of {\em coordinate polyrectangles}, that is 
those sets whose boundary is a 
closed polygonal curve with edges parallel to the coordinate axes.

Given a coordinate polyrectangle $E$, in what follows we 
will 
denote by $H^0$, 
$H^+$, and $H^-$ (respectively
$V^0$, $V^+$, and $V^-$) the sets of the horizontal (resp. vertical) edges
of $\partial E$ with zero, positive and negative $\varphi$--curvature. 
Moreover we set $H=H^0\cup H^+ \cup H^-$, and $V=V^0\cup V^+ \cup V^-$.

In what follows $\ell$ will denote the length of the edge $L$. 

\begin{rk}\label{r:polyode}
Given a coordinate polyrectangle $E$, the description of the variational 
crystalline curvature flow $E^\varepsilon(t)$ with forcing term 
$g_\varepsilon$ starting from $E$ will be obtained by combining the
calibrability properties of the edges proved in 
Section \ref{s:calibrability} and information about solutions of a
coupled system of ODEs solved by the coordinates of the vertices of $E(t)$
in any interval $I$ in which the number of vertices of $E^\varepsilon(t)$ does 
not change.

More precisely, let $(x_i(t),y_i(t))$, 
$i=1,\ldots,N$, be  
the coordinates of the vertices of the polyrectangles $E^\varepsilon(t)$, 
$t\in I$, ordered clockwise in such a way that
\[
(x_1,y_1)=(x_{N+1},y_{N+1}), \quad 
x_{2k}(t)=x_{2k+1}(t), \ 
y_{2k}(t)=y_{2k-1}(t), \quad k=1,\ldots, 
\frac{N}{2},\ t\in I, 
\]
and the edges of $E^\varepsilon(t)$ are given by
\[
L_i(t)=
\begin{cases}
[x_{i}(t), x_{i+1}(t)]\times \{y_i\} & \text{for $i$ odd}\\
\{x_i\}\times [y_{i}(t), y_{i+1}(t)] & \text{for $i$ even}.
\end{cases}
\]
Then $(x_1(t),\ldots,x_N(t),y_1(t),\ldots,y_N(t))$ is a solution 
in 
$I$ to the system of ODEs
\begin{equation}\label{f:odepoly}
\begin{cases}
x_{2k}'=x_{2k+1}'=-
\dfrac{2}{y_{2k}-y_{2k+1}} 
{\chi}_{L_{2k}}-g_\varepsilon(x_{2k}) \\[10pt]
y_{2k}'=y_{2k+1}'=
-\dfrac{1}{x_{2k}-x_{2k-1}} 
\left(2{\chi}_{L_{2k-1}}+ 
\displaystyle\int_{x_{2k-1}}^{x_{2k}}g_\varepsilon(s)\, 
ds 
\right)
\end{cases}
\quad k=1,\ldots \frac{N}{2}.
\end{equation}

The velocity field in \eqref{f:odepoly} is discontinuous
on the set
\[
\mathcal{D}:=\{(x_1,\ldots,x_N,y_1,\ldots,y_N)\ \text{such that}\ \exists 
i=1,\ldots, N 
\colon\ x_i\in \varepsilon \mathcal{I}\},
\]
so that the discontinuities only affect the motion
of the vertical edges.
We collect here the main features of the solutions to \eqref{f:odepoly}
(see \cite{Fi}, Chapter 2),
written in terms of motion of the edges of $\partial 
E^\varepsilon(t)$, 
$t\in I$. Since the system \eqref{f:odepoly} is autonomous, it is 
enough to 
discuss the properties of local solutions starting from given datum 
$E$.

\begin{itemize}
\item[(i)]
Pinning effect (stable discontinuities). Let
$\ell_p>0$ be the \textsl{pinning threshold} defined by
\[
\ell_p=
\begin{cases}
-\dfrac{2}{\alpha}, & \text{if}\ L\in V^+, \\[4pt]
0, & \text{if}\ L\in V^0, \\[4pt]
\dfrac{2}{\beta}, & \text{if}\ L\in V^-.
\end{cases} 
\]
Then every vertical edge $L\in \partial E$ with $\ell >\ell_p$ and such 
that either $\nu(L)=e_1$ and 
$L\in \varepsilon \mathcal{I}_{\beta,\alpha}$ or $\nu(L)=-e_1$ and 
$L\in \varepsilon \mathcal{I}_{\alpha,\beta}$ is pinned during the evolution
$E^\varepsilon(t)$ for every $t>0$ such that $\ell(t)\geq\ell_p$.

\item[(ii)] Trasversality condition. The motion of a vertical edge 
$L\in 
V^-\cup V^+$ with $\ell< \ell_p$ is uniquely determined until
$\ell(t)< \ell_p$. 

\item[(iii)] Uniqueness condition (unstable discontinuities). The 
uniqueness of 
the local solution 
starting 
from $E$
fails if and only if there is a vertical edge $L\in V^-\cup V^+$ with 
$\ell 
>\ell_p$ 
and such that either $\nu(L)=e_1$ and 
$L\in \varepsilon \mathcal{I}_{\alpha,\beta}$ or $\nu(L)=-e_1$ and 
$L\in \varepsilon \mathcal{I}_{\beta,\alpha}$ (unstable edges). 
If this does not occur, the 
solution is unique until the first time $t_0$ when $E^\varepsilon(t_0)$ has
a unstable edge.
\end{itemize}
\end{rk}

A significant class of coordinate polyrectangles with a well posed 
forced 
evolution is 
the following (see Proposition 
\ref{r:uniq} and Theorem \ref{t:comparison} below).

\begin{Definition}[$\mathcal{C}$--polyrectangles]
Given $\varepsilon>0$, we say that a coordinate polyrectangle 
$E$  is a \textsl{$\mathcal{C}$--polyrectangle} if every horizontal 
edge
of $E$ is a  $\mathcal{C}$--edge (see Definition \ref{d:cedge}).
\end{Definition}

\begin{rk}\label{r:pinnpoli}
By Proposition \ref{p:oninterf}, Proposition \ref{r:hcurvz}, and 
Remark \ref{r:polyode}, every 
$\mathcal{C}$-polyrectangle is calibrable, with velocity of 
an edge
$L$ given by
\begin{equation}\label{f:fcvel}
v_L=
\begin{cases}
\dfrac{\alpha+\beta}{2}+\chi_L\left(\dfrac{2}{\ell}-
\dfrac{(\beta-\alpha)\varepsilon}{4\ell}\right), & L\in H
\\[6pt]
\dfrac{2}{\ell}\chi_L+\gamma_\varepsilon, & L\in V^+\cup V^- , \ \ell 
< 
\ell_p
\\[6pt] 
0, & \text{$L\in V^0$, and $L\in V^- \cup V^+$ with $\ell \geq 
\ell_p$.}
\end{cases}
\end{equation}
(see Figure \ref{fig:fig14}), and the variational 
crystalline curvature flow with forcing term $g_\varepsilon$
of $E$ starts following the rules \eqref{f:odepoly}.
In particular, every edge $L\in V^0$ is pinned,
as well as any edge $L\in V^- \cup V^+$ with length 
$\ell\geq \ell_p$.
\end{rk}

\begin{figure}[h!]
\includegraphics[width=5.5cm]{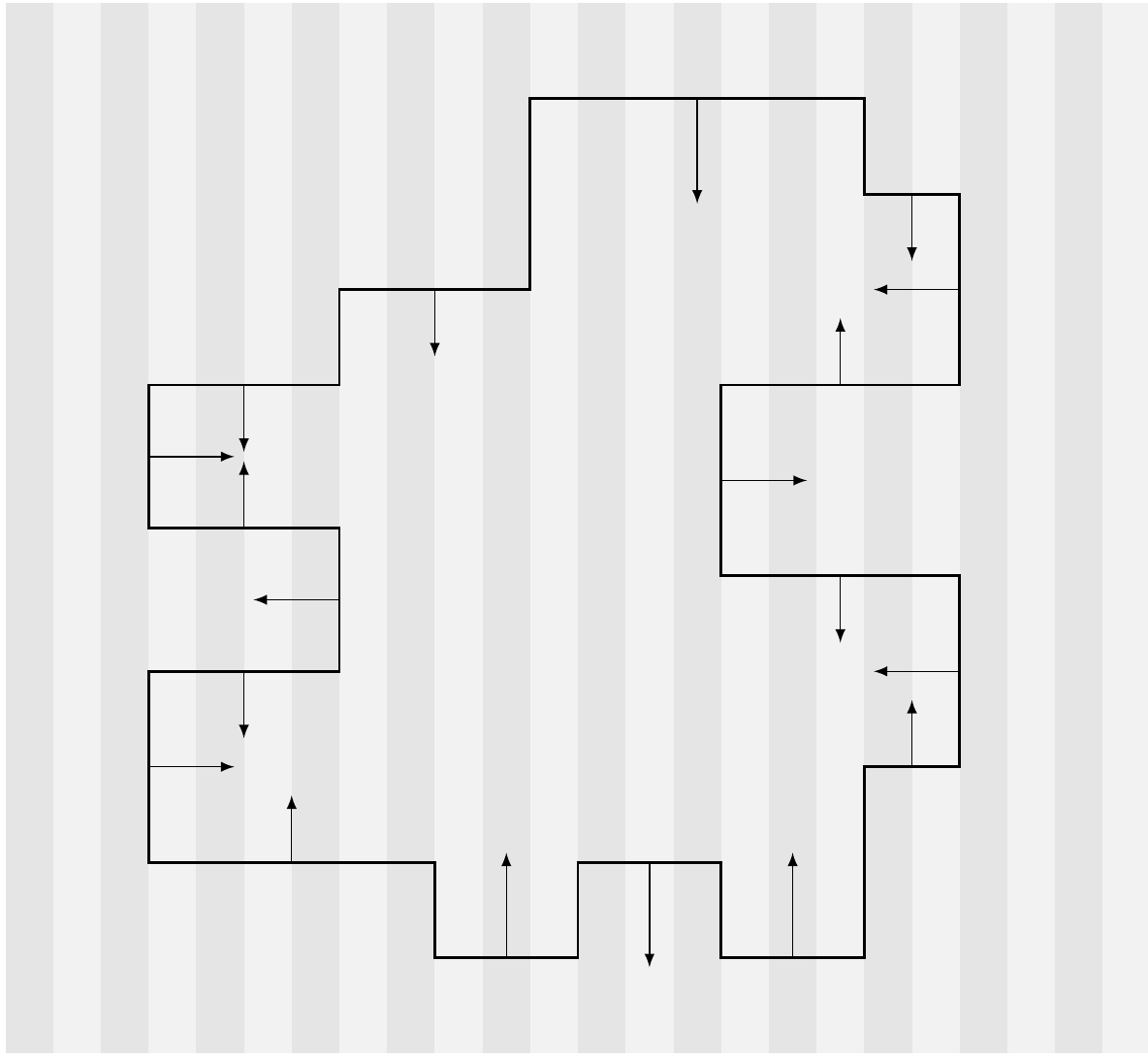}
\caption{A $\mathcal{C}$--polyrectangle with velocities of its edges  
(case $\alpha+\beta>0$) }\label{fig:fig14}
\end{figure}

\begin{teor}[Effective motion of coordinate polyrectangles]\label{t:effpoly}
Let $E_0$ be a coordinate 
polyrectangle.
For $\varepsilon>0$, let $E_0^\varepsilon$ be a coordinate 
polyrectangle
such that $d_H(E_0,E_0^\varepsilon)<\varepsilon$. 
Then there exists a variational crystalline curvature flow with forcing term 
$g_\varepsilon$
starting from $E_0^\varepsilon$.
Moreover, there exists a family of coordinate polyrectangles $E(t)$ such that 
every variational crystalline curvature flow $E^\varepsilon(t)$ with forcing 
term $g_\varepsilon$
of $E_0^\varepsilon$ 
converges to $E(t)$ in the Hausdorff topology and 
locally uniformly in $t$, for 
$\varepsilon \to 0$.
 
Denoting by $\ell(t)$ the length of an edge $L(t)\subseteq \partial 
E(t)$, the normal velocity $v_L(t)$
of $L(t)$ is given by
\begin{equation}\label{f:effevpoly}
v_L(t)=
\begin{cases}
\,\left\langle g+\dfrac{2}{\ell(t)} \right\rangle, & \text{if $L(t)\in V^+$, 
and}
\  \ell(t)<-2/\alpha, \\[8pt]
\,\left\langle g-\dfrac{2}{\ell(t)} \right\rangle, & \text{if $L(t)\in V^-$, 
and}
\  \ell(t)<2/\beta, \\[8pt]
\phantom{-}0, & \text{for the other vertical edges,} \\[4pt]
\dfrac{2}{\ell(t)}\chi_{L}+\dfrac{\alpha+\beta}{2}, & \text{if $L(t)$ is 
horizontal.}
\end{cases}
\end{equation}
The dynamics \eqref{f:effevpoly} is valid until $\ell(t)>0$. If an edge 
vanishes at $t=t_0$, the evolution proceedes reinitializing the ODEs by starting
from $E(t_0)$.
\end{teor}

\begin{proof}

The existence of a variational crystalline curvature flow with forcing 
term
$g_\varepsilon$ starting from a coordinate polyrectangle 
$E^\varepsilon_0$ and 
the fact that the effective evolution does not depend on the choice of 
the 
approximating data can
be obtained with arguments similar to the ones proposed in detail in 
Section 
\ref{s:rect}. We roughly sketch here the main features of such an evolution 
starting from a $\mathcal{C}$--polyrectangle
$E^\varepsilon_0$. 

By Remark \ref{r:pinnpoli}, the ``long'' vertical edges of $\partial 
E^\varepsilon_0$ with non zero 
$\varphi$--curvature, and 
all the vertical edges with zero $\varphi$--curvature are pinned. The  
vertical edges in $V^+$ with length $\ell<-2/\alpha$ move inward, while the 
vertical edges in $V^-$ with length $\ell<2/\beta$ move 
outward.

By Proposition \ref{r:hcurvz}, every edge $L\in H^0$ with pinned 
vertices has 
velocity $(\alpha+\beta)/2$, 
while every edge $L\in H^0$  with  an adjacent moving vertical edge breaks 
istantaneously
in a small part $L_\varepsilon$ with length $\varepsilon$, and a remaining part 
calibrated with velocity $(\alpha+\beta)/2$. During the evolution 
$L_\varepsilon$ shrinks and disappears in a time of order $\varepsilon$,
while the remaing part has pinned vertices and moves vertically with velocity 
$(\alpha+\beta)/2$. 

Similarly, every edge $L\in H^+\cup H^-$ with pinned vertices has velocity 
given by \eqref{f:fcvel}, 
while every edge $L\in H^+\cup H^-$ with  an adjacent moving vertical edge 
shrinks until the calibrability conditions of Propositions 
\ref{p:gencalibr} and \ref{p:gencalibrneg} hold, and possibly it breaks 
following the rules of Remark \ref{r:symm}.  

In any case, the possible ``breaking and recomposing" motion occurs in a lapse 
of 
time of order $\varepsilon$, so that the evolution $E^\varepsilon (t)$ 
can be 
approximated by a family 
of $\mathcal{C}$--polirectangles $\tilde{E}^\varepsilon (t)$ 
converging as $\varepsilon \to 0$ to a coordinate polyrectangle $E(t)$ 
in 
Hausdorff 
topology and locally uniformly in time. The effective 
velocities \eqref{f:effevpoly} of the edges of $E(t)$ are obtained  taking the 
limit, as $\varepsilon \to 0$ in \eqref{f:fcvel}.
In particular, the arguments for getting the velocities of the ``short''
vertical edges are the same of Case 1 in Section \ref{s:rect}.
 \end{proof}

The variational crystalline curvature flow with forcing term $g_\varepsilon$
starting from a $\mathcal{C}$--poly\-rectangle is unique and satisfies 
a 
comparison principle. 

\begin{prop}[Uniqueness]\label{r:uniq}
Given $\varepsilon>0$, the 
variational crystalline curvature flow $E^\varepsilon(t)$ with forcing 
term 
$g_\varepsilon$
starting from a $\mathcal{C}$--polyrectangle $E$ is unique.
\end{prop}

\begin{proof}
By Remark \ref{r:polyode}(iii), uniqueness may fail if and only if 
there exists
$t_0> 0$ such that $\partial E(t_0)$ has a unstable edge.
We will show that this never occurs when the initial datum $E$ is a 
$\mathcal{C}$--polyrectangle.

By Remarks  \ref{r:polyode} and  \ref{r:pinnpoli}, the evolution starts with 
all the vertical 
edges pinned on interfaces in 
$\varepsilon \mathcal{I}$ that are stable equilibria of the dynamics, except 
for the ``short" edges with nonzero $\varphi$--curvature. Moreover, the 
evolution may generate, for $t>0$ new vertical edges, due to the breaking 
phenomenon of the horizontal edges. Nevertheless, every new vertical edge
$L$ belongs to $V^0$ and it is pinned on a stable discontinuity of 
$g_\varepsilon$.

Hence, if we assume by contraddiction that there exists $t_0>0$ such 
that $\partial 
E(t_0)$ has a unstable edge, then it is the evolution 
$L(t_0)$ of a ``short" edge (that is $L\in V^-\cup V^+$ with $\ell< 
\ell_p$)
enlarging during the evolution.
More precisely, if we assume that $L(t_0)\in V^+$ and $\nu(L(t_0))=e_1$
(the other cases being similar), the following properties should be 
satisfied:
\[
L(t_0)\in\varepsilon\mathcal{I}_{\alpha,\beta},\quad 
\ell(t_0)>-2/\alpha,
\]
and there exists $\sigma>0$ such that
\[
g_\varepsilon =\alpha \ 
\text{on}\ L(t)
\ \text{for}\ t \in (t_0-\sigma, t_0), \quad 
\ell(t_0-\sigma)<-2/\alpha\ .
\]
Since 
$\partial E(t_0)$ needs to be calibrable, the horizontal edges adjacent to
$L(t_0)$ have either zero $\varphi$--curvature, length 
$\varepsilon/2$ and velocity $\beta$, or  positive $\varphi$--curvature, 
and length satisfying $\ell+\ell_\alpha-\ell_\beta \leq 
4/(\beta-\alpha)$ (see 
Proposition \ref{p:fracon}), and hence velocity
\[
v = \frac{2}{\ell}+\frac{\alpha+\beta}{2}+
\frac{\beta-\alpha}{2 \ell}
(\ell_\beta-\ell_\alpha) =
\frac{1}{2\ell}(4-(\beta-\alpha)(\ell+\ell_\alpha-
\ell_\beta)+2 
\beta\ell)> 0.
\]
In both cases
the horizontal edges adjacent to $L(t_0)$ have strictly positive 
velocity, 
in contraddiction with the fact that $L(t)$ enlarges for $t<t_0$ close 
enought 
to $t_0$.
\end{proof}

%

\begin{teor}[Comparison for forced flows]\label{t:comparison}
For a given $\varepsilon>0$,
let $E$ and $F$ be two $\mathcal{C}$--polyrectangles such that 
$E\subseteq F$, and let $E^\varepsilon(t)$, $F^\varepsilon(t)$ be the 
variational crystalline 
curvature 
flow of $E$ and $F$, respectively, with forcing term $g_\varepsilon$. 
Then
\begin{itemize}
\item[(i)] $E^\varepsilon(t)\subseteq F^\varepsilon(t)$ for every $t$.
\item[(ii)] the distance $d^\varepsilon(t)$ between $\partial 
E^\varepsilon(t)$ 
and $\partial F^\varepsilon(t)$
satisfies
\[
d^\varepsilon(t)\geq d^\varepsilon(0) \quad \forall t \geq 0, \qquad 
\text{and}\qquad 
d^\varepsilon(t_1)\geq d^\varepsilon(t_2)-\varepsilon \quad \forall 
t_1\geq t_2 
> 0.
\]
\end{itemize} 
\end{teor}

\begin{proof}
Given $\sigma>0$, let $E^\varepsilon_\sigma(t)$ be the 
variational crystalline curvature 
flow of $E$ with forcing term $g_\varepsilon+\sigma$.
Following the proof of the First Comparison Principle in \cite{GG}, 
mainly 
based on geometric arguments, we obtain that 
$E^\varepsilon_\sigma(t)\subseteq F^\varepsilon(t)$. 

Moreover, by
Theorem \ref{t:effpoly}, Proposition \ref{r:uniq}, and Theorem 2.8.2
in \cite{Fi}, we infer that for every $n\in\N$ 
there exists $\sigma_n>0$ such that, for every $t>0$, the evolutions 
$E^\varepsilon_{\sigma_n}(t)$  converge, as $n\to +\infty$, in 
Hausdorff 
topology to the variational crystalline curvature flow 
$E^\varepsilon(t)$ with 
forcing 
term $g_\varepsilon$ starting from $E$. Hence (i) follows from a 
passage to the 
limit for the inclusions 
$E^\varepsilon_{\sigma_n}(t)\subseteq F^\varepsilon(t)$.

In order to prove (ii), we underline that for every $t>0$ the minimal 
distance $d(t)$ between $\partial E^\varepsilon(t)$ and $\partial 
F^\varepsilon(t)$ is attained at points joined by a segment parallel 
to a 
coordinate axis, so that either $E(t)+d(t)e_1 \subseteq F(t)$ or  
$E(t)+d(t)e_2 
\subseteq F(t)$. On the other hand, since $E$ is a 
$\mathcal{C}$--polyrectangle, we have that, for every $\gamma\in\R$, 
$E^\varepsilon(t)+\gamma e_2$ is the variational crystalline curvature 
flow 
with 
forcing 
term $g_\varepsilon$ starting from $E+\gamma\, e_2$, and, for every 
$k\in\Z$,
$E^\varepsilon(t)+k\varepsilon e_1$  is the variational crystalline 
curvature 
flow 
with 
forcing 
term $g_\varepsilon$ starting from 
$E+k\varepsilon\, e_1$. 

Let $t_1 \geq 0$ be given. If $E(t_1)+d(t_1)e_2 \subseteq F(t_1)$, 
then, by (i) and the invariance of the flow under vertical 
translations,
we have that $E(t)+d(t_1)e_2 \subseteq F(t)$, and 
hence $d^\varepsilon(t_1)\geq d^\varepsilon(t)$, for every $t \geq 
t_1$. 
If, instead, $E(t_1)+d(t_1)e_1 \subseteq F(t_1)$, with the same 
argument 
we obtain that (ii) holds and that $d^\varepsilon(t_1)\geq 
d^\varepsilon(t)$
for every $t \geq t_1$ if and only if $d(t_1)=k\varepsilon$.

\end{proof}

\subsection{Evolution of more general sets}

The macroscopic effect of the underlying oscillating forcing term 
$g_\varepsilon$ on the crystalline curvature flow starting from any smooth,
connected bounded set $C_0$ may be captured in the following way.

For every $\varepsilon>0$, let $\mathcal{C}_\varepsilon$ be the set of 
all 
$\mathcal{C}$--polyrectangles. For every $E_\varepsilon(0) \in 
\mathcal{C}_\varepsilon$,  let $E_\varepsilon(t)$ be the unique variational 
crystalline curvature flow with forcing term $g_\varepsilon$ starting from
$E_\varepsilon(0)$. We define the families of sets
\[
C^+_\varepsilon(t)= \bigcap_{\substack{C_0 \subseteq P_\varepsilon(0) \\
P_\varepsilon(0)\in \mathcal{C}_\varepsilon }} P_\varepsilon(t), \qquad 
C^-_\varepsilon(t)= \bigcup_{\substack{Q_\varepsilon(0)\subseteq C_0 \\
Q_\varepsilon(0)\in \mathcal{C}_\varepsilon }} Q_\varepsilon(t),
\] 
and
\[
\begin{split}
C^+(t) & = \limsup_{\varepsilon \to 0^+}C^+_\varepsilon(t)=
\bigcap_{\eta>0}\left(\bigcup_{0<\varepsilon<\eta}C^+_\varepsilon(t)\right), 
 \\
C^-(t) & = \liminf_{\varepsilon \to 0^+}C^-_\varepsilon(t)=
\bigcup_{\eta>0}\left(\bigcap_{0<\varepsilon<\eta}C^-_\varepsilon(t)\right).
\end{split} 
\]
By Theorem \ref{t:comparison}, for every $P_\varepsilon(0)$, 
$Q_\varepsilon(0)\in \mathcal{C}_\varepsilon$ such that  
$Q_\varepsilon(0)\subseteq C_0 \subseteq P_\varepsilon(0)$, the evolutions 
satisfy $Q_\varepsilon(t)\subseteq P_\varepsilon(t)$ for every $t\geq 0$.
Hence we have
$C^-_\varepsilon(t)\subseteq C^+_\varepsilon(t)$, $t\geq 0$.

Moreover, denoting by $P(t),Q(t)$ the effective evolution starting 
from coordinate 
polyrectangles $P(0), Q(0)$, and evolving with the law 
\eqref{f:effevpoly}, by Theorem 
\ref{t:effpoly} we obtain that
\[
\bigcup_{Q(0)\subseteq C_0 } Q(t)\subseteq C^-(t)\subseteq C^+(t)\subseteq  
\bigcap_{C_0 \subseteq P(0)} P(t),
\qquad t\geq 0.
\]
Notice that $C^-(0)=C^+(0)=C_0$.
When $C^-(t)= C^+(t)=:C(t)$ for $t > 0$, this procedure defines the limit 
evolution starting from $C_0$ and driven, at a mesoscopic scale, by crystalline 
curvature flow with forcing term $g_\varepsilon$.

If $C_0$ is a bounded convex subset of $\R^2$ with nonempty interior, then 
$C^-(t)=C^+(t)$ for all $t\geq 0$, and
we 
can explicitely describe the motion $C(t)$ starting from $C_0$.
 
Approximation far from the "extreme points": constant vertical shift.
Let $\xi\in \partial C_0$ be such that $\nu(\xi)$ is not a coordinate 
vector.
In this case, we can choose approximating sequences of 
$\mathcal{C}$--polyrectangles with all edges with zero $\varphi$--curvature in 
a suitable neighborhood  of $\xi$ (see Figure \ref{fig:fig4}, left). 
The 
evolution by \eqref{f:effevpoly} 
of such 
an approximation is the following: the vertical edges are pinned, while the 
horizontal edges moves with velocity $(\alpha+\beta)/2$. At a macroscopic 
level, the effect is a vertical motion of $\partial C_0$ near $\xi$ with 
velocity $(\alpha+\beta)/2$ .
 
 \begin{figure}[h!]
 \includegraphics[height=3cm]{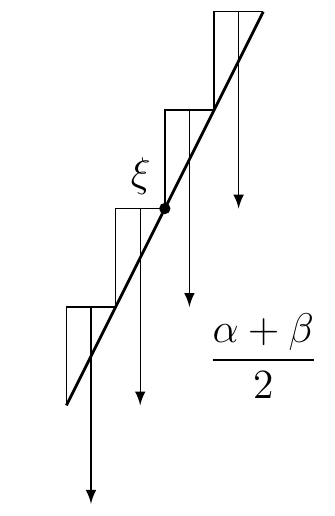}
 \qquad \qquad
 \includegraphics[height=3cm]{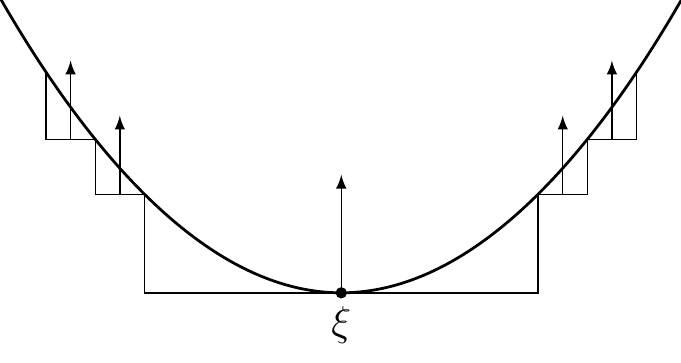}
 \caption{Approximation far from (left) and near to (right) an extreme 
 point}\label{fig:fig4}
 \end{figure}
 
Approximation near the "extreme points": flat evolution.
Let $\xi\in \partial C_0$ be such that $\nu(\xi)=\pm e_1$.
In this case, we can choose approximating sequences of 
$\mathcal{C}$--polyrectangles with one horizontal edge with positive 
$\varphi$--curvature in 
a suitable neighborhood  of $\xi$. Then the evolution $C(t)$, $t>0$, has a 
horizontal edge $L(t)$ with length $l(t)$ moving vertically with 
velocity 
$2/\ell(t) + 
(\alpha+\beta)/2$.
The same arguments show that the evolution $C(t)$, $t>0$, has flat vertical
edges ``generated by'' the points $\xi\in \partial C_0$ be such that 
$\nu(\xi)=\pm e_2$ and moving horizontally with velocity 
$H_g(\ell(t))$ 
 (see \eqref{f:harmean}).
 
 In conclusion, the effective evolution $C(t)$ of a convex set $C_0$ can be 
 depicted as follows.
 \begin{itemize}
 \item[-] The arcs with zero $\varphi$--curvature moves vertically with 
 velocity $(\alpha+\beta)/2$. Let us denote by $C_1(t)$, the set obtained with 
 this translation.  
 \item[-]
 There is an instantaneous generation of four flat edges parallel to the 
 coordinate axes and with the extreme points constrained on the set $C_1(t)$.
 The horizontal edges moves vertically with velocity $2/\ell + 
 (\alpha+\beta)/2$,
 while the vertical ones moves horizontally with velocity $H_g(\ell(t))$.
 \end{itemize}
 
 For example, the effective evolution starting from a circle is depicted in 
 Figure \ref{fig:figcircle}.
 
 \begin{figure}[h!]
 \includegraphics[height=4.5cm]{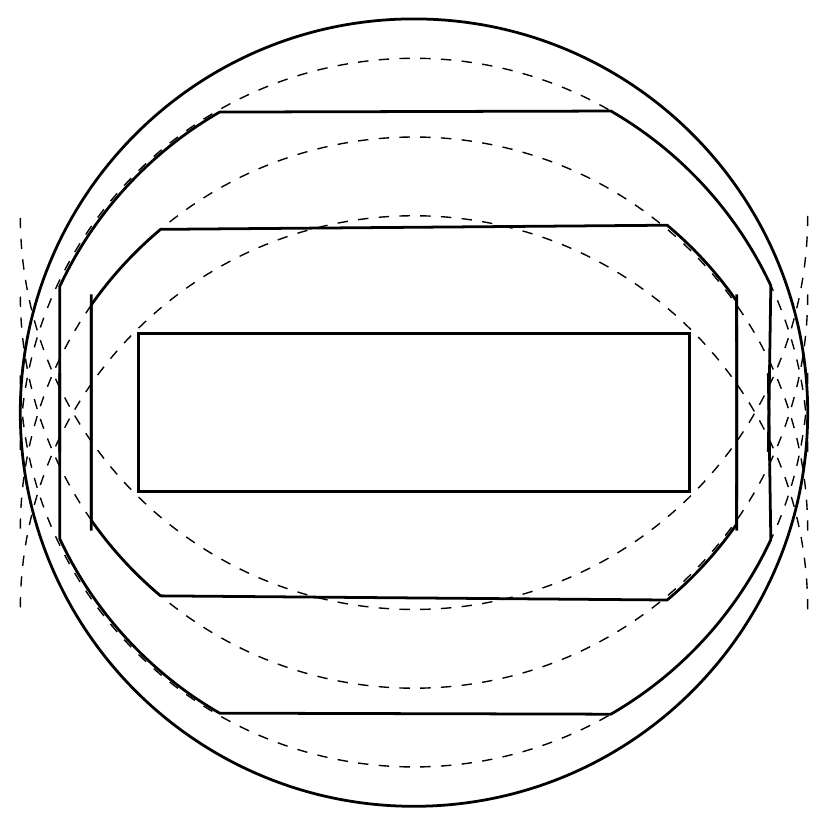}
 \caption{Effective evolution of a circle.}\label{fig:figcircle}
 \end{figure}


\end{document}